\documentclass[11pt,reqno,a4paper,twoside]{amsbook}
\usepackage[frenchb]{babel}
\usepackage{pb-diagram,graphicx,epsfig}
\usepackage{amssymb}

\newtheoremstyle{remarks}{3pt}{3pt}{}{}{\bfseries}{}{ }{}

\newtheorem{theorem}[equation]{\bf T{\footnotesize H\'EOR\`EME}}
\newtheorem{prop}[equation]{\bf P{\footnotesize ROPOSITION}}

\theoremstyle{definition}
\newtheorem{remark}[equation]{Remarque}

\newtheorem*{definition}{D\'efinition}

\theoremstyle{remarks}

\def\M{{\mathcal{M}}}
\def\C{{\mathcal{C}}}
\def\E{{\mathcal{E}}}
\DeclareMathAlphabet{\doba}{U}{msb}{m}{n}

\gdef\mR{\doba{R}}

\def\qed{{\leavevmode\unskip\nobreak\hfil\penalty 50\hskip 1em%
  \hbox{}\nobreak\hfil\lower 1pt\hbox{$\Box$\kern-.5pt}\parfillskip 0pt
  \finalhyphendemerits 0\par\bigbreak}}
\def\qedmath#1{\setbox0\hbox{$\displaystyle #1$}\templaenge=\textwidth\advance\templaenge by -\wd0%
\setbox1\hbox{$\Box$}\advance\templaenge by -2\wd1%
$$#1\hbox to0pt{\kern.5\templaenge$\Box$\kern-.5pt\hss}$$\par\bigbreak}

\let\<\langle     
\let\>\rangle

\def\al{{\alpha}}
\def\be{{\beta}}

\def\om{{\omega}}
\def\Om{{\Omega}}

\def\ep{{\varepsilon}}

\let\theta\vartheta

\def\phi{{\varphi}}

\def\nummerarray#1#2{\par\noindent\setbox0\hbox{\rm (#1)}\setbox1\hbox{$#2$}\unhcopy0%
\dimen0=.5\textwidth \advance\dimen0 by -\wd0 \advance\dimen0 by -.5\wd1 \kern\dimen0 \unhcopy1}

\def\ker{\mathop{{\rm ker}}}
\def\div{\mathop{{\rm div}}}

\def\eref#1{{\rm (\ref{#1})}} 

\def\res#1#2{{#1}\lower .11ex\hbox{$|$}\lower .644ex\hbox{$\scriptstyle #2$}}

\def\M{\mathcal{M}}

\begin{document}

\title[]{Relativit\'e g\'en\'erale (d'apr\`es M.Vaugon) et quelques probl\`emes math\'ematiques qui en sont issus.}
\author{Emmanuel Humbert\\
{\large Institut \'Elie Cartan de Nancy} \\
{\large Universit\'e H. Poincar\'e de Nancy I}\\
{\large BP 70239} \\
{\large 54506 Vandoeuvre-L\`es-Nancy Cedex}\\
{\large ehumbert@iecn.u-nancy.fr}}

\maketitle

Depuis deux ans, mon ancien directeur de th\`ese et ami Michel Vaugon a
entrepris de comprendre les intuitions physiques qui ont conduit
\`a l'axiomatique de la relativit\'e g\'en\'erale tout en gardant  un
langage de math\'ematicien et plus pr\'ecis\'ement de g\'eom\`etre. 
Ses notes, manuscrites, sont \`a mon sens d'une clart\'e remarquable et je
le remercie chaleureusement de me les avoir fournies. Cela
a \'et\'e l'occasion pour moi de comprendre ces notions, qui m'\'etaient
\'etrang\`eres bien que mes travaux de recherche aient des liens importants
avec le relativit\'e g\'en\'erale. J'ai commenc\'e \`a \'ecrire ce texte en suivant ses notes, non pas pour am\'eliorer son travail, mais pour m'approprier ces notions et les traduire dans mon propre langage. D'ailleurs, des cinq premiers chapitres, 
je ne
peux revendiquer qu'une part infime de la forme et quelques
remarques. M\^eme si la suite est plus personnelle, je tenais \`a ce que son nom apparaisse dans le titre parce que, sans lui, ce texte n'aurait pas vu le jour.
La version ici pr\'esente est tr\`es raccourcie. Elle sera en effet publi\'ee dans son int\'egralit\'e  chez Ellipses.  \\

\tableofcontents



 
\chapter{Quelques exp\'eriences fictives} \label{raisonnements}
Ce chapitre permet de montrer comment \`a partir d'observations simples, on
 peut rapidement faire du calcul relativiste et m\^eme aboutir \`a la
 c\'el\`ebre formule $E= mc^2$. Son but est uniquement culturel et ne servira pas dans la suite du texte. Le lecteur press\'e peut donc commencer la lecture directement au chapitre suivant.  \\

\noindent La th\'eorie de la relativit\'e est n\'ee d'une observation qui va \`a l'encontre de toutes nos intuitions : la vitesse de la lumi\`ere est la m\^eme par rapport \`a
n'importe quel observateur. C'est un fait observ\'e en 1887 par Michelson et Morley et qui se retrouve par le calcul (voir le Paragraphe \ref{equation_des_ondes} du Chapitre \ref{ondesg}).  Imaginons par exemple qu'un photon passe devant un observateur \`a $c$ km/h. Imaginons qu'un deuxi\`eme observateur aille exactement dans la m\^eme direction que le photon mais \`a $(c-1)$ km/h. Pour le sens commun, si ce deuxi\`eme observateur mesure la vitesse du photon, il doit trouver $1$ km/h. Or, exp\'erimentalement, il est d\'emontr\'e que cet observateur va trouver lui aussi une vitesse de $c$ km/h. En particulier, les lois habituelles de la m\'ecanique classique ne peuvent pas \^etre vraies, d'o\`u la n\'ecessit\'e de trouver un mod\`ele de l'espace-temps qui prenne en compte ce ph\'enom\`ene tout en gardant ``approximativement'' (c'est-\`a-dire pour tout ce qui se passe \`a l'\'echelle humaine) les lois de la m\'ecanique classique. \\

\noindent Dans ce premier chapitre, on explique comment, avant m\^eme de chercher un bon mod\`ele, on peut d\'eduire de cette observation plusieurs conclusions int\'eressantes gr\^ace \`a des raisonnements simples. \\

\section{Temps et longueurs}

\noindent Prenons un observateur $A$ qui se trouve dans un train qui avance \`a $v$ m\`etres/seconde par rapport au quai et dont les wagons ont une longueur de $l$ m\`etres. Prenons aussi un observateur $B$ qui regarde passer le train depuis le quai.
Maintenant, supposons qu'un photon parte de l'arri\`ere du wagon et qu'il
parcoure  ce wagon en un temps de $t$ secondes. Pour l'observateur $A$, le
photon a parcouru $l$ m\`etres en $t$ secondes soit une vitesse de $l/ t$
m\`etres/seconde.  
Maintenant, pour $B$ le photon a parcouru en $t$ secondes $l$ m\`etres plus la distance parcourue par le train en $t$ secondes soit 
$d= l + v t$. Ainsi pour l'observateur $B$, le photon va \`a  $(l+vt)/t= l/t + v$
m\`etres/seconde. Puisqu'on trouve des vitesses diff\'erentes pour les deux observateurs alors que l'exp\'erience dit au contraire qu'on doit trouver les m\^emes, c'est qu'il y a une erreur dans le raisonnement. 
En fait, on a consid\'er\'e que 
\begin{enumerate}
 \item le temps mesur\'e par $A$ et $B$ pour que le photon parcoure le wagon \'etaient les m\^emes (\'egaux \`a $t$).
\item La longueur du wagon mesur\'ee par $A$ et $B$ \'etait la m\^eme. 
\end{enumerate}

\noindent Pour arriver \`a un mod\`ele fid\`ele \`a la r\'ealit\'e, il faut
donc remettre en cause ces deux principes. Bien s\^ur, ces diff\'erences ne
se feront sentir qu'\`a des vitesses \'elev\'ees. Un observateur humain qui
observe ce qui se passe autour de lui ne se rendra pas compte de ces
diff\'erences de mesure.\\

\noindent Donc, pour la suite, on supposera que le temps entre deux
\'ev\'enements ou la longueur d'un objet d\'epend de l'observateur qui le
mesure. Se pose aussi le probl\`eme de la simultan\'eit\'e entre deux
\'ev\'enements qui 
d\'ependra elle aussi de  l'observateur. \\

\noindent {\bf Grandeurs conserv\'ees quel que soit l'observateur.} 
Comme on l'a expliqu\'e, on doit remettre en cause les notions de temps et de longueur mais on ne doit pas le faire n'importe comment. Par exemple, une longueur mesur\'ee perpendiculairement au d\'eplacement ne doit pas d\'ependre de l'observateur. En effet, supposons par exemple que ces longueurs se contractent quand la vitesse augmente (un raisonnement analogue se fait si on suppose que les distances  s'allongent) et reprenons le cas du train sur lequel se trouve l'observateur $A$ alors que l'observateur $B$ est rest\'e sur le quai. 

\begin{itemize} 
 \item Pla\c{c}ons nous  d'abord du point de vue de l'observateur $A$. Pour lui le train est immobile alors que les rails ont une vitesse $v$ non nulle. Donc la largeur des rails doit \^etre plus petite que l'\'ecartement des roues du train. Autrement dit, les roues du train laissent des traces \`a l'ext\'erieur des rails. 

\item Pour l'observateur $B$, c'est le train qui avance et donc l'\'ecartement de ses roues doit \^etre plus petit que l'\'ecartement des rails : les traces des roues du train doivent \^etre \`a l'int\'erieur des rails. 
\end{itemize}

\noindent Comme les traces laiss\'ees par le train ne peuvent pas \^etre \`a la fois \`a l'ext\'erieur et \`a l'int\'erieur des rails, c'est que l'\'ecartement des rails (ou des roues du train) doit \^etre le m\^eme pour les deux observateurs $A$ et $B$. \\

\noindent On pourrait imaginer une deuxi\`eme exp\'erience pour prouver que les distances mesur\'ees dans le sens du d\'eplacement ne d\'ependent pas non plus de l'observateur (mais comme on va le voir, ce raisonnement est faux) : les m\^emes observateurs $A$ et $B$ ont chacun une r\`egle gradu\'ee dans les mains. Au moment o\`u ils se croisent, l'observateur $A$ colle sa r\`egle sur celle de l'observateur $B$  plac\'ee dans le sens du d\'eplacement par exemple dans un bac \`a sable pos\'e sur le quai. 
Par le m\^eme genre de raisonnement que ci-dessus, on peut se dire qu'on arrive \`a une absurdit\'e (pour $B$ la marque laiss\'ee par sa r\`egle dans le sable doit \^etre plus grande que celle de $A$ et inversement). 
Il y a un probl\`eme dans cet argument : pour laisser une marque, l'observateur $A$ doit poser sa r\`egle dans le sable pendant un intervalle de temps certes tr\`es court mais non nul. Or pour $B$, cet intervalle de temps n'est pas le m\^eme, il est plus long.  En r\'esum\'e, dans cette exp\'erience, la marque laiss\'ee par la r\`egle de $A$ sera plus longue que celle laiss\'ee par la r\`egle $B$ mais pour deux raisons diff\'erentes: du point de vue de $A$, parce que sa r\`egle est plus longue et du point du vue de $B$ parce que $A$ a laiss\'e sa r\`egle un certain temps dans le sable.\\

\noindent {\bf Temps et distance pour deux observateurs.} 
On va maintenant imaginer deux exp\'eriences qui permettent de pr\'eciser les diff\'erences de mesure de temps et de distance dans le sens du d\'eplacement pour deux observateurs. 

\begin{enumerate}
 \item Reprenons toujours nos deux observateurs $A$ (dans le train) et $B$ (sur le quai). Notons $v$ la vitesse du train par rapport au quai. Imaginons qu'un photon fasse un aller-retour (plancher du wagon)-(plafond du wagon). Notons $h$ la hauteur du plafond par rapport au plancher du train ($h$ est la m\^eme pour $A$ et $B$ puisque c'est une distance qui est mesur\'ee perpendiculairement au d\'eplacement). 
 
\begin{itemize}
 \item $A$ mesure le temps $t_A$ pour cet aller-retour du photon. Pour lui le photon a parcouru la distance de $2h$. Donc la vitesse du photon est $c_1= \frac{2h}{t_A}$. 
\item $B$ mesure le temps $t_B$ pour le m\^eme trajet. Mais de son point de vue, le photon n'a pas un parcours vertical puisque le train avance. Plus pr\'ecisement en hauteur il a parcouru $2h$ et horizontalement $v t_B$. D'apr\`es Pythagore, il a parcouru $\sqrt{4h^2 + v^2 t_B^2}$ ce qui donne pour le photon une vitesse 
$c_2= \frac{\sqrt{4h^2 + v^2 t_B^2}}{t_B}$. 
\end{itemize}


\noindent Maintenant, comme la vitesse de la lumi\`ere est constante par rapport \`a n'importe quel observateur, on a $c_1=c_2=c$
 et on trouve que  
\begin{eqnarray*}
 \begin{aligned}
t_A& = \frac{2h}{c} \\
& = \frac{2h}{\sqrt{4h^2 + v^2 t_B^2}} t_B \\
& = \sqrt{ 1 - \frac{v^2 t_B^2}{v^2 t_B^2+4h^2} } t_B.
 \end{aligned}
\end{eqnarray*}
Comme $c =  \frac{\sqrt{4h^2 + v^2 t_B^2}}{t_B}$, on trouve que 
$$t_A= \sqrt{ 1 - \frac{v^2}{c^2}} t_B.$$

\noindent On admet donc la r\`egle suivante :\\

\noindent {\it \underline{R\`egle 1 :} Soient deux observateurs $A$ et $B$ qui se d\'eplacent \`a vitesse constante $v$ l'un par rapport \`a l'autre. On consid\`ere deux \'ev\'enements qui se passent {\bf au m\^eme endroit} pour  $A$. Alors si on note respectivement $t_A$ et $t_B$ les temps mesur\'es  entre ces deux \'ev\'enements par $A$ et $B$, on a 
$t_A= \sqrt{ 1 - \frac{v^2}{c^2}} \; t_B.$} \\

\begin{remark}
 Il est tr\`es important de noter que ces \'ev\'enements doivent se passer au m\^eme endroit pour l'un des observateurs (d'o\`u la n\'ecessit\'e de consid\'erer un aller-retour du photon). Sans cela, la r\`egle dirait aussi que $t_B= \sqrt{ 1 - \frac{v^2}{c^2}} \; t_A$, ce qui est faux bien s\^ur. Le point 1 de la Remarque \ref{erreur} ci-dessous illustre aussi la n\'ecessit\'e de consid\'erer de tels \'ev\'enements. 
\end{remark}

\noindent Avec le m\^eme raisonnement, on peut aussi en d\'eduire une r\`egle avec des hypoth\`eses un peu plus g\'en\'erales, qui nous servirons pour la suite  : \\

\noindent {\it \underline{R\`egle 1' :} Soient deux observateurs $A$ et $B$ qui se d\'eplacent \`a vitesse constante $v$ l'un par rapport \`a l'autre. On consid\`ere deux \'ev\'enements qui se passent \`a deux  endroits $X$ et $Y$ avec $(XY)$ perpendiculaire au mouvement pour  $A$ ({\bf attention, cette notion d\'epend de l'observateur}). Alors si on note respectivement $t_A$ et $t_B$ les temps mesur\'es  entre ces deux \'ev\'enements par $A$ et $B$, on a 
$t_A= \sqrt{ 1 - \frac{v^2}{c^2}} \; t_B.$} \\

\item Maintenant, consid\'erons que le photon fait un aller-retour (arri\`ere du wagon)-(avant du wagon). Cette fois, la longueur d\'epend de l'observateur. Notons $l_A$ (resp. $l_B$)  la longueur du wagon mesur\'ee par $A$ (resp. par $B$) et conservons $v$ pour sa vitesse. 

\begin{itemize}
 \item L'observateur $A$ mesure un temps $t_A$ pour cet aller-retour. La distance parcourue par le photon pendant ce temps est $2l$. Donc sa vitesse est 
$c= \frac{2l_A}{t_A}$. 
\item Pour $B$, s\'eparons le trajet aller du trajet retour. Notons  $t_B'$
  (resp. $t_B''$) le temps mesur\'e par $B$ pour  l'aller (resp. le
  retour). Pour $B$, la distance parcourue par le photon sur l'aller est
  $l_B+v t_B'$  (longueur du wagon plus distance parcourue par le wagon
  pendant le trajet aller). Pour le retour la distance est $l_B - vt_B''$.

%
\noindent On a donc 
\begin{eqnarray*} 
 c= \frac{l_B+ v t_B'}{t_{B'}} = \frac{l_B- v t_B''}{t_{B''}}.
\end{eqnarray*}
\end{itemize}
De cette \'equation, on tire 
$t_B' = \frac{l_B+ v t_B'}{c}$ c'est-\`a-dire $t_B'= \frac{l_B}{c-v}$. De m\^eme, 
$t_B''= \frac{l_B}{c+v}$. 
D'apr\`es la R\`egle 1 ci-dessus, on a 
\begin{eqnarray*}
 \begin{aligned}
 t_A & = \sqrt{1-\frac{v^2}{c^2} } (t_B'+t_B'') \\
& = \sqrt{1-\frac{v^2}{c^2} } \left(\frac{1 }{c-v} + \frac{1}{c+v} \right)l_B \\
& =  \frac{2}{\sqrt{c^2-v^2} } \; l_B.
 \end{aligned}
\end{eqnarray*}
Ainsi, 
$$l_A= \frac{c t_A}{2} = \frac{1}{\sqrt{1- \frac{v^2}{c^2}}} \; l_B.$$
On en d\'eduit la r\`egle suivante : \\

\noindent {\it \underline{R\`egle 2 :} Soient deux observateurs $A$ et $B$ qui se d\'eplacent l'un par rapport \`a l'autre \`a vitesse constante et soit $[PQ]$ un segment {\bf fixe pour $A$} et parall\`ele au mouvement. Alors, les distances $l_A$ et $l_B$ entre $P$ et $Q$ mesur\'ees respectivement par $A$ et $B$ sont li\'ees par 
$l_A = \frac{1}{\sqrt{1- \frac{v^2}{c^2}}} \;l_B.$}

\end{enumerate}

\begin{remark} \label{erreur} 

\noindent \\

\begin{enumerate}
 \item Encore une fois, pour appliquer la r\`egle qui lie les temps mesur\'es par deux observateurs, il faut bien v\'erifier que ces \'ev\'enements  se passent au m\^eme endroit pour l'un des observateurs. En effet, dans cette exp\'erience, les temps des trajets aller et retour mesur\'es par $A$ sont tous les deux de $t_A /2$. En appliquant la r\`egle de comparaison de temps, on a envie de dire que $t_A / 2= \sqrt{1-\frac{v^2}{c^2} } \; t_B'$ et $t_A / 2= \sqrt{1-\frac{v^2}{c^2} } \;t_B''$. Cela conduit \`a $t_B'= t_B''$, ce qui est faux (sinon, la vitesse du photon mesur\'e par $B$ n'est pas la m\^eme sur l'aller et le retour).  De m\^eme, dans la r\`egle de comparaison des longueurs, il est important que le segment $[PQ]$ mesur\'e soit fixe par rapport \`a l'un des observateurs.  
\item De ces raisonnements, on peut facilement d\'eduire des r\`egles de comparaison de temps lorsque deux \'ev\'enements ne se passent pas au m\^eme endroit ou de comparaison de longueur pour un segment quelconque. Ces r\`egles sont appel\'ees {\it transformations de Lorentz}.
\end{enumerate}
\end{remark}

\section{Masse et impulsion} 
En m\'ecanique classique, consid\'erons un objet $q$ de masse $m$ anim\'e dans un rep\`ere (c'est-\`a-dire pour un observateur galil\'een donn\'e) d'une vitesse $\overrightarrow{v}$. 

\begin{definition}
Le vecteur $\overrightarrow{p} := m \overrightarrow{v}$ est appel\'e {\it quantit\'e de mouvement } ou {\it impulsion} de l'objet $q$. 
\end{definition} 

\noindent Soit maintenant un syst\`eme compos\'e de $n$ objets $q_1,\cdots,q_n$ de quantit\'es de mouvement respectives $\overrightarrow{p}_1, \cdots, \overrightarrow{p}_n$. Alors, une loi fondamentale de la dynamique en m\'ecanique classique est \\

\noindent {\bf Loi de conservation de quantit\'e de mouvement :} {\it la quantit\'e de mouvement du syst\`eme d\'efini par $\overrightarrow{p} = \overrightarrow{p}_1 + \cdots + \overrightarrow{p}_n$ est constante avec le temps. En particulier, elle est conserv\'ee lors des chocs entre les objets du syst\`eme. }

\noindent On va supposer que cette loi, bas\'ee sur l'exp\'erience, est toujours valide en relativit\'e. Mais pour cela, il va falloir pr\'eciser les choses, puisque les notions de temps et de distance sont remises en cause. Dans ce but, nous reprenons nos deux observateurs $A$ (dans le train) et $B$ (sur le quai). Chacun va lancer une pierre sph\'erique ($q_A$ pour $A$ et $q_B$ pour $B$) de masse $m$ perpendiculairement au mouvement et avec une vitesse $V$ de mani\`ere \`a ce que les deux pierres se rencontrent. Pour l'observateur $B$ on travaille  dans un rep\`ere $(Bxy)$ o\`u l'axe $(Bx)$ est parall\`ele aux rails et $(By)$ est orthogonal aux rails. Pour l'observateur $A$, on travaille dans le rep\`ere $(Axy)$ dont seule l'origine est diff\'erente : on prend $A$ \`a la place de $B$. Autrement dit, le rep\`ere $(Axy)$ se d\'eplace \`a la vitesse constante $v$ par rapport au rep\`ere $(Bxy)$. \\

\noindent {\bf Avant le choc}, 

\begin{itemize} 
\item L'observateur $A$ voit la pierre $q_A$ anim\'ee de la  vitesse (exprim\'ee dans $(Axy)$)
\[ \overrightarrow{v}^A_A = \left( 
\begin{array}{c} 
0 \\ 
V 
\end{array} \right) \] 
tandis que l'observateur $B$ voit la m\^eme pierre anim\'ee de la vitesse (exprim\'ee dans $(Bxy)$) 
\[ \overrightarrow{v}^B_A = \left( \begin{array}{c} v \\ V'  \end{array} \right) \] 
o\`u $v$ est la vitesse du train et $V'$ est \`a calculer. 
\item Par sym\'etrie de la situation, l'observateur $A$ voit la pierre $q_B$ anim\'ee de la vitesse
\[ \overrightarrow{v}^A_B = \left( \begin{array}{c} v \\ -V' \end{array} \right) \]  
tandis que l'observateur $B$ voit la m\^eme pierre anim\'ee de la vitesse 
\[ \overrightarrow{v}^B_B = \left( \begin{array}{c} 0 \\ -V \end{array} \right) \] 
(on a bien s\^ur arbitrairement choisi une orientation des axes).
\end{itemize}


\noindent {\bf Calcul de $V'$ :} $V'$ est la composante perpendiculaire au train du vecteur vitesse de la pierre $q_A$ vue par $B$. Regardons le temps que met la pierre $q_A$ pour parcourir une distance $l$ sur l'axe perpendiculaire aux rails (l'axe de $y$. On oublie les composantes dans l'autre direction). D'apr\`es le paragraphe pr\'ec\'edent, cette distance ne d\'epend pas de l'observateur. Notons $t_A$ (resp. $t_B$) le temps mesur\'e par $A$ (resp. par $B$) pour que la pierre parcoure cette distance $l$. D'apr\`es la R\`egle 1' du paragraphe pr\'ec\'edent, on a $t_A= \sqrt{ 1 - \frac{v^2}{c^2}} \; t_B$. D'autre part, on a $V = \frac{l}{t_A}$ et $V'= \frac{l}{t_B}$. On obtient ainsi que 
\begin{eqnarray} \label{V'=}
V'= V   \sqrt{ 1 - \frac{v^2}{c^2}}.
\end{eqnarray}  

\noindent {\bf Apr\`es le choc :} on consid\`ere bien s\^ur que le choc est \'elastique. On remarque que les composantes des vitesses selon l'axe des $x$ n'est pas modifi\'ee. Regardons maintenant les composantes des vitesses sur l'axe des $y$. Par sym\'etrie de la situation, $A$ et $B$ doivent voir revenir leur propre pierre \`a la m\^eme vitesse. Si cette vitesse est diff\'erente de celle avant le choc (par exemple strictement sup\'erieure \`a $V$) alors le syst\`eme a gagn\'e de l'\'energie ce qui va \`a l'encontre des lois physiques. Donc $A$ et $B$ doivent voir revenir leur pierre avec la vitesse $V$. Autrement dit,  

\begin{itemize} 
\item L'observateur $A$ voit la pierre $q_A$ anim\'ee de la vitesse 
\[ \overrightarrow{w}^A_A = \left( \begin{array}{c} 0 \\ -V \end{array} \right) \] 
tandis que l'observateur $B$ voit la m\^eme pierre anim\'ee de la vitesse 
\[ \overrightarrow{w}^B_A = \left( \begin{array}{c} v \\ -V'  \end{array} \right), \]
\item l'observateur $A$ voit la pierre $q_B$ anim\'ee de la vitesse
\[ \overrightarrow{w}^A_B = \left( \begin{array}{c} v \\ V' \end{array} \right) \]  
tandis que l'observateur $B$ voit la m\^eme pierre anim\'ee de la vitesse 
\[ \overrightarrow{w}^B_B = \left( \begin{array}{c} 0 \\ V \end{array} \right). \] 
\end{itemize}

\noindent Pour l'observateur $A$, v\'erifions si la loi de conservation de quantit\'e de mouvement de la m\'ecanique classique est toujours valable pour le syst\`eme $(q_A,q_B)$. On doit avoir 
$$m   \overrightarrow{v}^A_A + m \overrightarrow{v}^A_B = m\overrightarrow{w}^A_A + m\overrightarrow{w}^A_B.$$
Or ce n'est pas le cas puisque la deuxi\`eme coordonn\'ee donne $m(V-V')= m(V'-V)$. D'o\`u vient l'erreur de raisonnement ? La r\'eponse est simple : on a consid\'er\'e que les masses des pierres $q_A$ et $q_B$ \'etaient les m\^emes du point de vue des observateurs $A$ et $B$. Or il semble raisonnable de remettre en cause ce principe. En effet, supposons qu'un objet conserve la m\^eme masse quelle que soit sa vitesse. Il suffira d'une quantit\'e finie d'\'energie pour acc\'el\'erer la particule \`a n'importe quelle vitesse choisie, y compris \`a une vitesse plus grande que celle de la lumi\`ere. Or, d'apr\`es les formules trouv\'ees au paragraphe pr\'ec\'edent, la vitesse de la lumi\`ere est une barri\`ere infranchissable. L'un des moyens d'expliquer ce fait est de supposer que la masse d'un objet tend vers l'infini quand sa vitesse tend vers celle de la lumi\`ere. \\

\noindent Revenons \`a notre probl\`eme. En m\'ecanique relativiste, on va toujours supposer que la quantit\'e de mouvement est le produit de la masse par la vitesse mais on va aussi supposer que la masse d'un objet vu par un observateur d\'epend de sa vitesse par rapport \`a cet obervateur. 
Notons $\overrightarrow{V}$ ce vecteur vitesse.  On va dire que sa masse est une fonction de $\overrightarrow{V}$ not\'ee 
$m_{\overrightarrow{V}}$. On va chercher cette fonction de mani\`ere \`a ce que 
la quantit\'e de mouvement soit conserv\'ee apr\`es le choc c'est-\`a-dire de mani\`ere \`a avoir l'\'egalit\'e 
\begin{eqnarray} \label{qm}
m_{\overrightarrow{v}^A_A}   \overrightarrow{v}^A_A +  m_{\overrightarrow{v}^A_B}  \overrightarrow{v}^A_B = m_{\overrightarrow{w}^A_A} \overrightarrow{w}^A_A + m_{\overrightarrow{w}^A_B} \overrightarrow{w}^A_B. 
\end{eqnarray}
On remarque que cette \'egalit\'e est v\'erifi\'ee si pour une vitesse $\overrightarrow{V}$ et un objet de masse $m$ (pour donner un sens \`a la masse d'un objet, il faut comprendre ce terme comme \'etant sa masse au repos, c'est-\`a-dire \`a vitesse nulle), on a 
$$m_{\overrightarrow{V}}= \frac{m}{\sqrt{1 - \frac{\| \overrightarrow{V} \|}{c^2}}}  $$
(o\`u $\| \overrightarrow{V} \|$ est la norme euclidienne de $\overrightarrow{V}$).
En effet, l'\'egalit\'e de la premi\`ere coordonn\'ee est clairement v\'erifi\'ee. Par ailleurs, 
pour la deuxi\`eme coordonn\'ee, en utilisant \eref{V'=}, on a 
$$\frac{V}{ \sqrt{1 -\frac{V^2}{c^2}}} - \frac{ V'}{\sqrt{1 - \frac{v^2 + (V')^2}{c^2}}} = 0 = 
- \frac{V}{ \sqrt{1 -\frac{V^2}{c^2}}} + \frac{ V'}{\sqrt{1 - \frac{v^2 + (V')^2}{c^2}}}.$$

\noindent Cette hypoth\`ese faite sur la masse semble valable dans la mesure o\`u elle tend vers $+ \infty $ si la vitesse tend vers $c$. \\

\noindent En r\'esum\'e, on retiendra : \\

\noindent {\it Soit $q$ un objet de masse $m$ (au repos) vu par un observateur $A$. On suppose que $q$ a une vitesse $\overrightarrow{V}$ par rapport \`a $A$. Alors sa masse vue par $A$ est 
\begin{eqnarray} \label{masse_rel} 
m_{\overrightarrow{V}}= \frac{m}{\sqrt{1 - \frac{\| \overrightarrow{V} \|}{c^2}}}
\end{eqnarray}
et sa quantit\'e de mouvement est 
$$\overrightarrow{p}= \frac{m \overrightarrow{V}}{\sqrt{1 - \frac{\| \overrightarrow{V} \|}{c^2}}}.  $$
Avec ces d\'efinitions, on garde la loi de conservation de quantit\'e de mouvement \'enonc\'ee ci-dessus dans le cadre de la m\'ecanique classique. }

\noindent Le raisonnement ci-dessus est, d'un point de vue math\'ematique en tout cas, beaucoup moins rigoureux que ceux donn\'es dans le paragraphe pr\'ec\'edent pour \'etablir les R\`egles 1, 1' et 2. N\'eanmoins, les mesures physiques permettent de v\'erifier ces lois avec une grande pr\'ecision.

\section{La formule $E=mc^2$.}
L'{\it \'energie} est par d\'efinition la capacit\'e d'un syst\`eme � modifier un \'etat, \`a produire un travail entra\^{\i}nant un mouvement, de la lumi\`ere ou de la chaleur. 
La formule $E=mc^2$ appara\^{\i}t d\'ej\`a dans les travaux de Poincar\'e et dit qu'une particule au repos poss\`ede de par sa masse, une \'energie interne due aux forces d'interaction entre particules. L'intuition provient de la remarque exp\'erimentale suivante : si un corps \'emet une \'energie (par exemple par rayonnement) $E$, on mesure que sa masse diminue de $\frac{E}{c^2}$ d'\`ou l'id\'ee que cette masse $m$ se soit transform\'ee en \'energie avec la relation 
$m = \frac{E}{c^2}$. 
L'\'etude de cette grandeur physique joue un r\^ole fondamentale en relativit\'e en raison de la loi suivante : \\

\noindent {\bf Loi de conservation de l'\'energie :} {\it l'\'energie totale d'un syst\`eme qui n'a pas d'\'echange avec l'ext\'erieur est constante avec le temps.} \\

\noindent En particulier, tout comme l'impulsion d\'efinie dans le paragraphe pr\'ec\'edent, cette quantit\'e est une int\'egrale premi\`ere du syst\`eme ce qui se d\'efinit parfaitement en math\'ematiques. \\

\noindent Maintenant essayons de d\'eduire la formule $E=mc^2$ de la discussion pr\'ec\'edente. D'abord, il para\^{\i}t naturel de penser que l'\'energie au repos d'une particule doit \^etre proportionnelle \`a sa masse. Autrement dit, pour une particule de masse au repos $m_0$, on a $E= k m_0$ pour $k \in \mR$ si la particule est immobile. Changeons maintenant d'observateur. Si ce nouvel observateur mesure l'\'energie de la particule, il doit trouver la m\^eme valeur \`a laquelle s'ajoute l'\'energie cin\'etique de la particule. Par contre, le r\'esultat trouv\'e sera toujours proportionnel \`a sa masse $m$ observ\'ee. On a aussi envie de dire que ce coefficient de proportionnalit\'e doit \^etre universel. On fera donc l'hypoth\`ese suivante : 
l'\'energie totale d'une particule de masse $m$ mesur\'ee par un observateur (li\'ee \`a sa masse au repos par la formule \eref{masse_rel}) est de la forme $km$ o\`u $k \in \mR$. 
On vient de dire que l'\'energie totale de la particule \'etait son \'energie interne (i.e. son \'energie au repos) \`a laquelle s'ajoute son \'energie cin\'etique. En d'autres termes, on a $E=km = km_0+ E_c$ ou encore 
\begin{eqnarray} \label{energie_tot}
E_c= k (m -m_0).
\end{eqnarray}
   
\noindent {\bf Calcul de $k$ :}  
nous aurons besoin de deux lois fondamentales de la m\'ecanique classique  :

\begin{enumerate} 
\item Soit $q$ un objet de masse $m$. Notons $\overrightarrow{a}(t)$ son acc\'el\'eration \`a l'instant $t$. 
Alors $m \overrightarrow{a}(t) = \overrightarrow{F}$ o\`u $\overrightarrow{F}(t)$ est la r\'esultante des forces qui s'appliquent \`a $q$ \`a l'instant $t$. 
\item La diff\'erence d'\'energie cin\'etique (c'est-\`a-dire uniquement due \`a sa vitesse) de l'objet $q$ entre deux instants est \'egale au travail de la force $\overrightarrow{F}$ qui s'applique sur $q$ le long de sa trajectoire. 
\end{enumerate}

\noindent La deuxi\`eme loi n'a, a priori, aucune raison d'\^etre remise en question en relativit\'e.  Par contre,  la premi\`ere loi n'est pas satisfaisante puisque la masse d\'epend du temps. On remarque cependant que $m \overrightarrow{a}(t)$ n'est autre que la d\'eriv\'ee de la quantit\'e de mouvement en fonction du temps. Il para\^{\i}t plus naturel de garder cette formule en relativit\'e :
$\frac{d}{dt}\overrightarrow{p}(t) = \overrightarrow{F}(t)$.    
On rappelle que le travail d'une force sur une trajectoire $c:[a,b] \to \mR^3$ est donn\'e par 
$$\int_a^b (\overrightarrow{F}(c(t)), c'(t) ) dt$$
o\`u $\overrightarrow{F}(c(t))$ est la force qui s'applique en $c(t)$ et o\`u $(\cdot,\cdot)$ est le produit scalaire euclidien. 
Soit donc un objet $q$ de masse au repos $m_0$ soumis \`a une force
$\overrightarrow{F}$ constante. \`A l'instant $t=0$ supposons que cette particule est
au repos. On note  $E_c$ l'\'energie cin\'etique de la particule \`a
l'instant $t=1$, $\overrightarrow{p}(t)$ la quantit\'e de mouvement \`a l'instant $t$,
$m(t)$ la masse \`a l'instant $t$ et $v(t)$ la vitesse \`a l'instant
$t$. On remarque que le vecteur vitesse est en tout point proportionnel \`a
$\overrightarrow{F}$. Avec les lois 1 et 2, on obtient   

\begin{eqnarray*}
\begin{aligned} 
E_c & =  \int_0^1 (\overrightarrow{F},c'(t)) dt \\
& = \int_0^1 (m v)' v dt \; \hbox{ car } \| c'(t) \| = v(t) \\
& = \int_0^1 (mv^2)' - m v' v dt \\
& =  m(1) v^2(1) - \int_0^1 m v' v dt.
\end{aligned}
\end{eqnarray*}
En utilisant la valeur de la masse trouv\'ee ci-dessus
 $$E_c = \frac{m_0 }{\sqrt{1 - \frac{v(1)^2}{c^2}}} v(1)^2 - \int_0^1 \frac{m_0}{\sqrt{1 -\frac{v^2}{c^2} } } v' v dt.$$
 Posons maintenant $u = v(t)$ dans l'int\'egrale ci-dessus. On a alors 
\begin{eqnarray*}
\begin{aligned}
E_c & =   \frac{m_0 }{\sqrt{1 - \frac{v(1)^2}{c^2}}} v(1)^2 - \int_0^{v(1)} \frac{m_0 u}{\sqrt{1 -\frac{u^2}{c^2} } }  du \\ 
& =  \frac{m_0 }{\sqrt{1 - \frac{v(1)^2}{c^2}}} v(1)^2 + \left[  c^2 {m_0}\sqrt{1 -\frac{u^2}{c^2} } \right]_0^{v(1)} \\
& = \frac{m_0 }{\sqrt{1 - \frac{v(1)^2}{c^2}}} c^2 - m_0 c^2\\
& = (m(1) -m(0)) c^2.
\end{aligned}
\end{eqnarray*}

\noindent En comparant ce r\'esultat avec \eref{energie_tot}, on obtient que $k=c^2$ ce qui donne que l'\'energie totale d'une particule au repos de masse $m$ est $E=mc^2$.

\chapter{Mod\'elisation de l'espace-temps} \label{modelisation}

\section{En m\'ecanique classique} \label{modelisation_mc}
\smallskip

\noindent {\bf Mod\'elisation } { \it En m\'ecanique classique, l'univers est mod\'elis\'e par un espace affine 
$\M$ de dimension $4$ muni d'une forme quadratique $T$ sur $E:= \overrightarrow{\M}$
de signature $(+,0,0,0)$.}  

\begin{remark} On pr\'ef\`ere prendre un espace affine plut\^ot que $\mR^4$,
  ce qui \'evite d'avoir un point base et des directions privil\'egi\'ees. 
\end{remark}

\noindent {\it Orientation en temps: }
Comme les vecteurs isotropes ($T(v) = 0$ ) forment un hyperplan de $E$,
l'ensemble  
$E \setminus \{
v \in E |T(v) = 0 \} $ a exactement deux composantes connexes. Choisissons
l'une d'elles une fois pour toutes et notons-la $E^+$. Ce sont les
directions dites {\it positives}.  \\

\noindent Dans tout le paragraphe, $b$ d\'esignera la forme bilin\'eaire sym\'etrique associ\'ee \`a $T$.  

\begin{definition} 
\noindent \,

\begin{enumerate}
\item Un {\it observateur} est une courbe de genre temps  i.e. telle qu'il
  existe une param\'etrisation $c: I \to \M$ ($I$ est un intervalle r\'eel)
  tel que pour tout $t \in I$, $T(c'(t)) \not=0$.  

\item Un {\it observateur galil\'een} est une droite non isotrope. 
\end{enumerate}
\end{definition}

\noindent Consid\'erer les observateurs galil\'eens parmi tous les
observateurs est naturel pourtant, physiquement cela pose un probl\`eme.
Cela signifie qu'il y a des observateurs privil\'egi\'es dans
l'univers. Qui sont-ils ? Il faut remarquer que si l'on conna\^{\i}t un
observateur galil\'een, on les conna\^{\i}t tous.  \\

\noindent Choisissons maintenant un produit scalaire $g$ sur $\ker(T)= \{
v\in E | b(v,x)=0 \, \forall x \in E \, \}$. Notons $\| \cdot\|$ la norme
associ\'ee. On peut d\'efinir naturellement:

\begin{definition}
\noindent \, 

\begin{enumerate}
\item Soit $A,B \in \M$. On dit que $A$ et $B$ sont {\it simultan\'es} si $\overrightarrow{AB} \in \ker (T)$. "\^Etre simultan\'es" est une relation d'\'equivalence dont les classes sont de la forme $A + \ker(T)$. Ce sont des hyperplans affines qui physiquement, repr\'esentent l'univers \`a un instant donn\'e. 
\item Lorsque $A,B \in \M$ sont simultan\'es, on peut calculer leur distance: 
$$d(A,B) = \| \overrightarrow{AB} \|.$$ 
\item {\it Le temps qui s\'epare} $A,B \in \M$ est donn\'e par $\tau_{AB} = \sqrt{T(\overrightarrow{AB})}$. Autrement dit, deux points sont simultan\'es si et seulement si $\tau_{AB}=0$. 
\end{enumerate}
\end{definition}

\noindent Consid\'erons un observateur galil\'een $D$ dirig\'e par un
vecteur $i_D$ unitaire ($T(i_D)=1$) et orient\'e positivement. Si l'on fixe
une origine $A \in D$ (on notera $D_A$), on a un
isomorphisme naturel 
\[ \phi_{D_A} : \left| \begin{array}{ccc}
\M & \to&  \ker(T) \times \mR \\
B & \mapsto & (\overrightarrow{v}, t) 
\end{array} \right. \]
o\`u $\overrightarrow{v}, t$ sont d\'etermin\'es par l'\'ecriture unique $\overrightarrow{AB} = \overrightarrow{v} + t i_D$. 
Prendre une origine consiste \`a d\'efinir pour $D$ un temps $t=0$. Pour
$D_A$, l'{\it univers observable} \`a l'instant $t$ est
$\phi_{D_A}^{-1}(\ker(T) \times \{t \})$.\\ 

\noindent {\bf Param\'etrisation normale positive d'un observateur} 

\noindent(on dit aussi {\it param\'etrisation par le temps}).

\noindent Soit $D$ un observateur. Soit $c:I\to \M$ ($I$ est un intervalle de $\mR$) une param\'etrisation de $D$ telle que $T(c'(t))$ ne s'annule jamais sur $I$. Quitte \`a remplacer $c(t)$ par $c(-t)$, on peut supposer qu'en tout point $c'(t) \in E^+$. Posons 
$s(t) = \int_{t_0}^t \sqrt{T(c'(u))} du$ o\`u $t_0$ est un point fix\'e de $I$.   On voit que $s$ est un diff\'eomorphisme de $I$ sur l'intervalle $J:= s(I)$.
Posons maintenant $C = c \circ s^{-1} $. On voit que pour tout $t \in J$,
$$T(C'(t)) = 1 \, \hbox{ et } \, C'(t) \in E^+.$$
En effet, 
$$T(C'(t)) = T \big( c'(s^{-1}(t))  (s^{-1})'(t) \big) = T \big( \frac{c'(s^{-1}(t))}{s'(s^{-1}(t))} \big).$$ 
Le r\'esultat est maintenant clair puisque $s'(t) = \sqrt{T(c'(t))}$. 
Une telle param\'etrisation de $D$ est appel\'ee {\it param\'etrisation
  normale positive} de l'observateur $D$. On a montr\'e qu'une telle
param\'etrisation existe toujours et 
on remarque facilement qu'elle est unique \`a translation en temps pr\`es. \\

\noindent {\bf Vitesse} Soit $D, \tilde{D}$ deux observateurs et $\al,\tilde{\al}$ des param\'etrisations normales positives respectives de $D$ et $\tilde{D}$. Quitte \`a faire une translation en temps, on peut supposer que pour $t$ fix\'e, $\al(t)$ et $\tilde{\al}(t)$ sont simultan\'es.  Alors, 
$\tilde{\al}'(t)$ s'\'ecrit de mani\`ere unique 
$$\tilde{\al}'(t)= \overrightarrow{k} + a \al'(t)$$
o\`u $\overrightarrow{k} \in \ker T$ et o\`u $a \in \mR$. Comme $T(\tilde{\al}'(t))=T(\al'(t)) = 1$, on voit que $a=1$.

\begin{definition}
Le vecteur $\overrightarrow{k}$ est appel\'e {\it vecteur vitesse} de $\tilde{D}$ par rapport \`a $D$ et est not\'e
$\overrightarrow{v}_{\tilde{D}/D}$ 
\end{definition}

\noindent {\it Remarques et propri\'et\'es:} 
  \begin{enumerate}
  \item La vitesse ainsi d\'efinie d\'epend de l'instant $t$.  
\item Si $D, \tilde{D}$ sont des observateurs galil\'eens alors le
    vecteur vitesse $\overrightarrow{v}_{\tilde{D}/D}$ est constant en
    fonction du temps. Cette d\'efinition est bien conforme \`a l'id\'ee que l'on se fait de la vitesse  : le quotient de la distance par le temps.  Prenons en effet deux observateurs galil\'eens  $D$
    et $\tilde{D}$. Prenons $A, B\in D$ et $\tilde{A}, \tilde{B} \in
    \tilde{D}$ tels que $A$ et $\tilde{A}$ sont simultan\'es ainsi que $B$
    et $\tilde{B}$. Naturellement, on voit que la norme du  vecteur vitesse $\overrightarrow{v}_{\tilde{D}/D}$ est \'egal \`a 
 $$\frac{\| \overrightarrow{B\tilde{B}} - \overrightarrow{A \tilde{A} } \|}{\tau_{AB}}$$ 
 (i.e. distance / temps).  
\item On a $\overrightarrow{v}_{\tilde{D}/D} = - \overrightarrow{v}_{D/ \tilde{D}}$. 
\item Si $\tilde{\tilde{D}}$ est un troisi\`eme observateur, on a 
$$\overrightarrow{v}_{\tilde{\tilde{D}}/D} = \overrightarrow{v}_{\tilde{\tilde{D}}/\tilde{D}} +\overrightarrow{v}_{\tilde{D}/D}.$$
    \end{enumerate}

\noindent {\bf Acc\'el\'eration:} Reprenons les notations utilis\'ees pour la d\'efinition de la vitesse. On d\'efinit l'acc\'el\'eration de $\tilde{D}$ par rapport \`a $D$ par 
$\overrightarrow{a}_{\tilde{D}} = \frac{d}{dt} \overrightarrow{v}_{\tilde{D}/D}.$  
Supposons que $D$ est galil\'een. Puisque la vitesse relative de deux observateurs galil\'eens est constante, l'acc\'el\'eration d\'efinie ci-dessus {\bf ne d\'epend pas de l'observateur galil\'een $D$}.

\section{En relativit\'e restreinte} \label{modelisation_rr}
On abandonne la m\'ecanique classique pour la raison suivante: de mani\`ere exp\'eri--mentale, on constate que la vitesse de la lumi\`ere est constante (en norme) par rapport \`a n'importe quel observateur. Supposons qu'un observateur $D$ voie passer la lumi\`ere dans une certaine direction \`a la vitesse $c\simeq 300 000$ km/h et qu'un   autre observateur $D'$ ait une vitesse de $(c -1)$ km/h par rapport \`a $D$ dans la m\^eme direction que la lumi\`ere. La propri\'et\'e $(4)$ du vecteur vitesse d\'efini en m\'ecanique classique implique que la vitesse de la lumi\`ere par rapport \`a $D'$ sera de $1$ km/h, ce qui contredit l'exp\'erience. On doit donc abandonner le mod\`ele de la m\'ecanique classique. C'est ainsi qu'est n\'ee la relativit\'e restreinte en 1905 gr\^ace aux travaux d'Einstein.

\noindent {\bf Mod\'elisation } { \it En relativit\'e restreinte, l'univers est mod\'elis\'e par un espace affine 
$\M$ de dimension $4$ muni d'une forme quadratique $T$ sur $E:= \overrightarrow{\M}$
de signature $(-,+,+,+)$.}  

\noindent D'un point de vue physique, on prend un espace affine pour \'eviter qu'il y ait des points privil\'egi\'es. Malgr\'e tout, dans la pratique, on se placera la plupart du temps dans l'{\it espace de Minkowski} 
$(\mR^4,\eta)$ 
o\`u 
$$\eta:= -dt^2 + (dx^1)^2 + (dx^2)^2 + (dx^3)^2.$$  

\noindent Dans $E$, il y a trois types de vecteurs $\overrightarrow{v}$:
 \begin{itemize}
 \item les vecteurs de type {\it temps}: $T(\overrightarrow{v}) < 0$;
 \item les vecteurs de type {\it lumi\`ere}: $T(\overrightarrow{v}) = 0$;
 \item les vecteurs de type {\it espace}: $T(\overrightarrow{v}) > 0$.
 \end{itemize}

\noindent \\

\noindent   On peut comme en m\'ecanique classique choisir une orientation de
  temps. En effet, l'ensemble $E$ des vecteurs de type temps a deux
  composantes connexes. Il faut remarquer que cet ensemble est un c\^one
  dont le bord est l'ensemble des vecteurs de type lumi\`ere. 
On  choisit l'une des composantes connexes et on la note $E^+$, cet ensemble repr\'esentant l'ensemble des vecteurs de type temps orient\'es positivement. Une autre mani\`ere de voir les choses est de fixer un vecteur $\overrightarrow{v}_0$ de type temps et de dire qu'un vecteur $\overrightarrow{v}$ est orient\'e positivement si $g(\overrightarrow{v},\overrightarrow{v}_0) >0$, o\`u $g$ est la forme bilin\'eaire associ\'ee \`a $T$. \\

\noindent De la m\^eme mani\`ere, on d\'efinit 
\begin{itemize}
\item un {\it observateur}:  courbe de genre temps (i.e. dont tout vecteur tangent est de type temps)
 \item un {\it observateur galil\'een}: droite de type temps. 
 \end{itemize}
 \noindent Soit $D$ un observateur. En proc\'edant comme en m\'ecanique
 classique, on montre qu'il existe une param\'etrisation $c:I \to \M$ de $D$, unique
 \`a translation en temps pr\`es  que l'on appellera {\it param\'etrisation
   normale positive} qui v\'erifie
$T(c'(t))= -1$ et $c'(t) \in E^+$ pour tout $t \in I$. 
On verra plus loin qu'il y a d'autres param\'etrisations normales
naturelles. \\

\noindent Soit $D$ un observateur param\'etr\'e par $c: I \to \M$. 

\noindent En m\'ecanique classique, l'{\it espace vu par $D$} ou {\it univers
  observable pour $D$} \`a l'instant $t$ est l'ensemble des points
simultan\'es \`a $c(t)$ c'est-\`a-dire $c(t)+\ker(T)$. L'espace vu par $D$
ne d\'epend pas de $D$ mais seulement du point $c(t)$. 
 
\noindent En relativit\'e restreinte, on a une d\'efinition analogue : 
\begin{definition} on appelle {\it espace vu par $D$ au point $c(t)$}
l'espace affine $c(t) + [c'(t)]^\perp$. En particulier, {\bf cet espace
  d\'epend de l'observateur $D$} et pas seulement de $c(t)$. Physiquement,
cela correspond \`a l'ensemble des points simultan\'es \`a l'observateur
\`a un instant donn\'e. 
\end{definition} 
Si $A,B \in \M$, cela n'a pas de sens de se demander si $A$ et $B$
sont simultan\'es. Par contre si $A \in D$ et si $B \in \M$, on peut se
demander si $B$ est simultan\'e \`a $A$ pour $D$. C'est le cas si 
$B \in
c(t) + [v]^\perp$ o\`u $v$ est un vecteur tangent \`a $D$ en $A$. 
Contrairement \`a ce qui se passe en m\'ecanique classique, la
simultan\'eit\'e n'est pas sym\'etrique (si $A \in D$, $A' \in D'$ et si
$A'$ et $A$ sont simultan\'es pour $D$, ils ne le sont pas forc\'ement pour
$D'$). \\

\noindent Pour avoir une bonne image en t\^ete, le plus simple est d'imaginer 
 $\mR^2$ muni
de $- dt^2+dx^2$. Soit $D$ un observateur galil\'een. Si $D$ est
parall\`ele \`a l'axe des abscisses, l'espace vu par $D$ est vertical
(axe des ordonn\'ees). Si l'espace $\overrightarrow{D}$ (droite vectorielle associ\'ee \`a $D$) se rapproche de la position limite $x=t$ alors
$[\overrightarrow{D}]^\perp$ aussi ($[\overrightarrow{D}]^\perp$ est le sym\'etrique de $\overrightarrow{D}$ par rapport \`a $x=t$).


\noindent Ce mod\`ele est assez pratique pour visualiser correctement ce qui se passe, mais pour coller
plus \`a la r\'ealit\'e physique, il faudrait plut\^ot penser \`a $\mR^2$
muni de $-dt^2 +  \ep^2 dx^2$ avec $\ep$ petit. Ainsi, lorsque $\ep$ est
suffisamment petit, l'orthogonal de toute droite de type temps est
``presque'' verticale. En effet, l'orthogonal d'une droite dirig\'ee par
$(a,b)$ est dirig\'e par $(\ep,a/b)$ et on se rapproche du mod\`ele $\mR^2$ muni de
$-dx^2$ qui permet de visualiser la m\'ecanique classique (consid\'erer
$dt^2$ ou $-dt^2$ ne change rien \`a la g\'eom\'etrie). \\

\begin{definition}  Soit $D$ un observateur (pas forc\'ement galil\'een) 
param\'etr\'e par $c:I \to \M$.
\begin{itemize}
\item Soient  $A=c(t_1), B=c(t_2) \in D$. Le {\it temps propre pour l'observateur $D$}
entre $A$ et $B$ est donn\'e par 
$$\tau_{AB} = \int_{t_1}^{t_2} \sqrt{-T(c'(t))} dt.$$    
Physiquement, il s'agit du temps que mesure l'observateur $D$ 
entre $A$ et  $B$. 
\item Soient  $A,B \in \M$ simultan\'es pour $D$ (i.e. il existe $t$ tel que
  $A,B \in c(t) + [c'(t)]^\perp$). On d\'efinit la distance de $A$ \`a $B$ (pour $D$) par  $d(A,B)=  \sqrt{T(\overrightarrow{AB})}$. 
\end{itemize}  
\end{definition}

\begin{remark}
\noindent \, 

\begin{enumerate}
\item Le temps propre ne d\'epend pas de la param\'etrisation choisie. 
\item La d\'efinition est la m\^eme qu'en m\'ecanique classique (on
  remplace juste $T$ par $-T$):  en m\'ecanique classique,  lorsque l'on 
\'ecrit 
$$\tau_{AB} = \int_{t_1}^{t_2} \sqrt{T(c'(t))} dt$$
($T$ est ici de signature $(+,0,0,0)$)
on voit que l'on obtient $\tau_{AB}= \sqrt{T(\overrightarrow{AB})}$. En effet,
on peut \'ecrire de mani\`ere unique $\overrightarrow{Ac(t)} =  \overrightarrow{k}(t) + \alpha(t) \overrightarrow{AB}$ o\`u
$\overrightarrow{k}(t) \in \ker(T)$ et $\alpha(t) \in \mR$. 
Alors $c'(t) = \overrightarrow{k}'(t) + \alpha'(t)
\overrightarrow{AB}$ et  $\alpha'(t)$ est de signe constant (sinon $c$ n'est
pas de genre temps). Supposons par exemple $\alpha'(t)>0$. Alors, 
$\sqrt{T(c'(t))} = \alpha'(t) \sqrt{  T(\overrightarrow{AB})}$. 
Ainsi, $\tau_{AB} = (\alpha(t_2) - \alpha(t_1))  \overrightarrow{AB}$.
En revenant \`a la d\'efinition de $\alpha$, on voit que $\alpha(t_1)=0$
et $\alpha(t_2)= 1$ (car $c(t_1)=A$ et $c(t_2)=B$), d'o\`u le r\'esultat.

\item Si l'on se donne $A,B \in \M$, cela n'a pas de sens comme en
  m\'ecanique classique  de parler de
  temps qui s\'epare $A$ et $B$. Cela d\'epend de la trajectoire
  choisie. Imaginons $A,B \in \M$ tels que $ \overrightarrow{AB}$ est de
  type temps. Prenons la trajectoire directe (i.e. la droite $(AB)$). On
  regarde le temps propre entre $A$ et $B$ et l'on trouve
  $\tau_{AB}= \sqrt{-T(\overrightarrow{AB})}$. Maintenant, imaginons une trajectoire entre $A$ et $B$ de type
  lumi\`ere,  (ou du moins tr\`es proche d'une trajectoire de type
  lumi\`ere) et calculons le temps propre entre $A$ et $B$. On param\'etrise par $c:I \to \mR$. On voit que
  $T(\overrightarrow{c'(t)}) \equiv 0$ si bien que $\tau_{AB}=0$. \'Evidemment,
  physiquement, aucun observateur ne peut suivre une courbe de type
  lumi\`ere mais si la trajectoire s'en rapproche, le temps propre $\tau_{AB}$ sera tr\`es petit. 
 En particulier, on remarque que
{\bf si deux observateurs $D$ et $D'$ ont une trajectoire qui passent par $A$
  et $B$ et si $D$  voyage \`a une vitesse proche de celle de la
  lumi\`ere, i.e. avec une trajectoire dont la tangente se rapproche de la
  position limite ``lumi\`ere'', $D$ aura un temps propre beaucoup plus
  petit que $D'$ entre $A$ et $B$}. Cette  propri\'et\'e, contraire \`a
l'intuition,  est connue sous le nom de {\it
  paradoxe de langevin}. On la pr\'esente habituellement en disant que deux
jumeaux sont n\'es sur Terre. L'un part en voyage \`a une vitesse proche de
celle de la lumi\`ere. Quand il revient sur Terre, il est beaucoup plus
jeune que son fr\`ere.


 \item Soit $\overrightarrow{v}$ un vecteur de type temps. Alors
   l'hyperplan vectoriel  $[\overrightarrow{v}]^\perp$ est de
   type espace (i.e. $T_{/ [\overrightarrow{v}]^\perp}$ est de signature
   $(+,+,+)$). En effet, la signature de $T$ est obtenue en ajoutant un $-$
   (car $\overrightarrow{v}$ de type temps) 
   \`a celle de $T_{/ [\overrightarrow{v}]^\perp}$. 
En particulier, dans la d\'efinition de $d(A,B)$, $T(
\overrightarrow{AB})>0$. 
\item Dans un sens, tout est beaucoup plus naturel qu'en m\'ecanique
  classique car il n'y a pas besoin de se donner un produit scalaire
  suppl\'ementaire. Toute l'information est contenue dans $T$. 
\end{enumerate}
\end{remark}

\noindent {\bf Param\'etrisation normale positive pour un observateur galil\'een $D$.} 
Soit $D$ un observateur galil\'een de vecteur directeur $i_D$ unitaire ($T(i_D)=-1)$ orient\'e positivement) param\'etr\'e par $c: I \to \M $ tel que pour tout $t$, $c'(t) = i_D$. 
Pour $D$, l'unit\'e de temps est
$i_D$. Maintenant consid\'erons un autre observateur $\tilde{D}$ param\'etr\'e par $\tilde{c}:\tilde{I}
\to \M$. Il est naturel de d\'ecomposer pour tout $t$
$\tilde{c}'(t) = \overrightarrow{k} + \alpha i_D$ o\`u $\alpha \in \mR$ et $\overrightarrow{k} \in
[i_D]^\perp$. En m\'ecanique classique, si $T(\tilde{c}'(t))= 1$ alors $\alpha = 1$. En relativit\'e restreinte, $T(\tilde{c}'(\tilde{t}))=-1$ n'implique pas $\alpha=1$. 
En particulier, il est naturel de consid\'erer une param\'etrisation
orient\'ee positivement pour laquelle $\alpha=1$ pour tout $t$, c'est-\`a-dire qui respecte l'unit\'e de temps pour $D$ (voir la Remarque \ref{rem_param} ainsi que le point 1 de la Remarque \ref{rem_vitesse}). Une telle
param\'etrisation existe,  est unique \`a translation en temps pr\`es
(m\^eme argument que pour l'existence des autres param\'etrisations
normales) et sera appel\'ee {\it param\'etrisation
  normale positive pour l'observateur $D$}. 
  
 \begin{remark}\label{rem_param} L'une des propri\'et\'es d'une telle param\'etrisation est la suivante:
  si $t \in I$, $\tilde{t} \in \tilde{I}$ sont  tels que 
$c(t)$ et $\tilde{c}(\tilde{t})$ sont simultan\'es pour $D$, alors pour tout $a$, $c(t+a)$  et $\tilde{c}(\tilde{t}+a)$ sont simultan\'es pour $D$. En effet, quitte \`a faire une translation en temps, on peut supposer $\tilde{t}=t$ et alors  
$$\tilde{c}(t+a) - c(t+a) = (\tilde{c}(t) -c(t)) + \int_t^{t+a} (\tilde{c}'(s) -c'(s)) ds \in [i_D]^\perp$$
car les deux termes du membres de droites sont dans $[i_D]^{\perp}$. On voit avec cet argument pourquoi il est n\'ecessaire de d\'efinir ce type de param\'etrisation relativement \`a un observateur galil\'een. 
Par contre, si $c$ et $\tilde{c}$ sont des param\'etrisations normales
positives de deux observateurs $D$ et $\tilde{D}$ et si $c(t)$ et $\tilde{c}(t)$ sont simultan\'es pour $D$, en g\'en\'eral {\bf $c(t+a)$ et $\tilde{c}(t+a)$ ne sont pas simultan\'es pour $D$.} 
\end{remark}

\noindent {\bf Vitesse:}
Soit $D$, $\tilde{D}$ deux observateurs param\'etr\'es  respectivement par $c:I \to \M$ et $\tilde{c}: \tilde{I} \to \M$. On suppose que $c$ est une  param\'etrisation normale positive. Fixons un point $A \in D$ que l'on \'ecrit $A=c(t)$. Quitte \`a faire une translation en temps, on peut supposer que $\tilde{c}(t)$ est simultan\'e \`a $c(t)$ pour $D$. 
Le vecteur $\tilde{c}'(t)$ s'\'ecrit de mani\`ere unique 
$$\tilde{c}'(t) = \overrightarrow{k} +\alpha c'(t)$$
o\`u $\overrightarrow{k}\in [c'(t)]^\perp$ et o\`u $\alpha \in \mR$.
\begin{definition} 
Le {\it vecteur vitesse} de $\tilde{D}$ par rapport \`a $D$ est 
$$\overrightarrow{v}_{\tilde{D}/D} = \frac{\overrightarrow{k}}{\alpha}.$$
\end{definition}

\begin{remark} \label{rem_vitesse}
\noindent \, 

\begin{enumerate}
\item Si $\tilde{c}$ est une param\'etrisation normale pour $D$, $\alpha=1$ et donc 
$\overrightarrow{v}_{\tilde{D}/D} = \overrightarrow{k}.$
\item Parler de vitesse sans pr\'eciser l'observateur par rapport auquel on se place n'a pas de sens sauf pour la vitesse de la lumi\`ere qui est constante par rapport \`a n'importe quel observateur (voir Proposition \ref{v_rel} ci-dessous). 
\item La vitesse relative de deux observateurs galil\'eens est constante. Avec les notations ci-dessus, $\overrightarrow{v}_{\tilde{D}/D}$
correspond exactement \`a la vitesse relative de deux observateurs galil\'eens dirig\'es par $c'(t)$ et $\tilde{c}'(t)$.
\item Cette d\'efinition correspond bien \`a l'intuition. Si $D$ et $D'$ sont deux observateurs galil\'eens qui se croisent en $A$ et si $B \in D$ et $B' \in D'$ sont simultan\'es pour $D$ alors
$$\overrightarrow{v}_{\tilde{D}/D} = \frac{\overrightarrow{BB'}}{\tau_{AB}}$$
(=distance parcourue dans la direction de $\overrightarrow{BB'}$ divis\'ee par le temps).
En effet, soit $\overrightarrow{v}$, $\overrightarrow{v}'$ des vecteurs directeurs  de $D$ et $D'$ normaux orient\'es positivement. \'Ecrivons $B=A+a\overrightarrow{v}$ et $B'= A + a'\overrightarrow{v}'$. Comme $B$ et $B'$ sont simultan\'es pour $D$, on a $\overrightarrow{BB'} \in [v]^\perp$. On \'ecrit $\overrightarrow{v}'= \overrightarrow{k}+s\overrightarrow{v}$ o\`u $\overrightarrow{k} \in [v]^\perp$ et o\`u $s \in \mR$. On a alors $a'\overrightarrow{v}' -a\overrightarrow{v} \in [v]^\perp$ i.e. $a' \overrightarrow{k} + (a's-a) \overrightarrow{v} \in [v]^\perp$ et donc $a's-a=0$ ou encore $a'= a/s$.
Cela donne que 
\begin{eqnarray} \label{bb'}
\overrightarrow{BB'}=a' \overrightarrow{k}.
\end{eqnarray} 
Par ailleurs, on a  par d\'efinition 

\noindent $$\overrightarrow{v}_{\tilde{D}/D}=  \frac{\overrightarrow{k}}{s} = \frac{a' \overrightarrow{k} }{a}= \frac{\overrightarrow{BB'}}{a}.$$
Comme $a= \sqrt{-T(\overrightarrow{AB})}= \tau_{AB}$ on a le r\'esultat cherch\'e. 
\end{enumerate}
\end{remark}

\noindent Avec cette mod\'elisation de l'espace, on a 
\begin{prop} \label{v_rel}
La vitesse de la lumi\`ere par rapport \`a n'importe quel observateur est constante. 
\end{prop} 

\noindent Dans  cette proposition, par "vitesse de la lumi\`ere", il faut bien \'evidemment comprendre "norme du vecteur vitesse de la lumi\`ere".

\begin{remark}
Avec la normalisation choisie, on trouve que la vitesse de la lumi\`ere est $1$. Pour changer cette valeur, il suffit de normaliser les vecteurs de type temps \`a une autre constante que $1$. 
\end{remark} 

\begin{proof}
Soit $D$ un observateur param\'etr\'e par $c:I \to \M$, param\'etrisation normale positive
et $L$ un rayon de lumi\`ere i.e. une droite de type lumi\`ere. 
Soit $\overrightarrow{l}$ un vecteur directeur de $L$ (qu'on ne peut pas normaliser puisque $T(\overrightarrow{l})=0$) on peut param\'etrer $L$ par $\tilde{c}(t)= M + t \overrightarrow{l}$.  Soient $A=c(t)$ un point de $D$ et $B =\tilde{c}(\tilde{t})$ un point de $L$ simultan\'e \`a $c(t)$ pour $D$. On \'ecrit 
$$\tilde{c}'(\tilde{t})= \overrightarrow{l} = \overrightarrow{k} + a c'(t)$$
o\`u $a \in \mR$ et o\`u $\overrightarrow{k} \in [c'(t)]^\perp$. Quitte \`a remplacer $\overrightarrow{l}$ par $-\overrightarrow{l}$ on peut supposer que $a>0$. Par d\'efinition, on a 
$$\overrightarrow{v}_{L/D} = \frac{\overrightarrow{k}}{a}.$$  
Remarquons que 
$$0= g(\overrightarrow{l},\overrightarrow{l}) = g(\overrightarrow{k},\overrightarrow{k}) + a^2 g(c'(t),c'(t)).$$
 Comme $g(c'(t),c'(t))=-1$ et comme $a>0$, on a 
 $a = \sqrt{T(\overrightarrow{k})}$, ce qui implique que  
 $$\overrightarrow{v}_{L/D} = \frac{\overrightarrow{k}}{a}.$$
 Ainsi 
 $$\|\overrightarrow{v}_{L/D}\| = \sqrt{T(\overrightarrow{v}_{L/D})} = 1.$$
 ce qui d\'emontre le r\'esultat.  

\end{proof}

\section{En relativit\'e g\'en\'erale} \label{modelisation_rg}
On abandonne la relativit\'e restreinte principalement parce que, comme on
le verra dans le prochain chapitre, elle n'est pas adapt\'ee \`a la
description du comportement de la mati\`ere. Un autre probl\`eme est qu'en
relativit\'e restreinte, comme en m\'ecanique classique, les observateurs
galil\'eens sont des observateurs privil\'egi\'es, ce qui physiquement
n'est pas satisfaisant. Einstein a ainsi introduit la th\'eorie de la relativit\'e g\'en\'erale, dont il a publi\'e les bases en 1915.

\noindent {\bf Mod\'elisation } { \it En relativit\'e g\'en\'erale, l'univers est mod\'elis\'e par une vari\'et\'e $\M$ munie d'une m\'etrique lorentzienne $g$, c'est-\`a-dire une m\'etrique de signature $(-,+,+,+)$ sur chaque espace tangent $T_x\M$.}


\begin{remark} On simplifiera en prenant des vari\'et\'es $C^\infty$ mais Hawking a \'etudi\'e les cons\'equences de consid\'erer des vari\'et\'es de r\'egularit\'e plus faible. 
\end{remark}

\noindent Soit $\overrightarrow{v} \in T \M$. On dit que 
\begin{itemize}
\item $v$ est de genre {\it temps} si $g(\overrightarrow{v},\overrightarrow{v}) <0$;
\item $v$ est de genre {\it lumi\`ere} si $g(\overrightarrow{v},\overrightarrow{v}) =0$;
\item $v$ est de genre {\it espace} si $g(\overrightarrow{v},\overrightarrow{v}) >0$.
\end{itemize}
Une courbe est de genre temps (resp. lumi\`ere, resp. espace) si en tout point ses vecteurs tangents sont de type temps (resp. lumi\`ere, resp. espace). \\

\noindent Dans chaque espace tangent $T_x\M$, l'ensemble 
$E_x:= \{ \overrightarrow{v} \in T_x \M| g_x(v,v) <0\}$ a deux composantes connexes.

\begin{definition}
Une {\it orientation en temps continue} de $(\M,g)$ est une orientation en temps de chaque espace tangent (i.e. le choix d'une composante connexe $E_x^+$ de $E_x$) telle que pour tout $x \in \M$, il existe un voisinage $V_x$ de $x$ et un champ de vecteur $X \in \Gamma(TV_x)$ sur $V_x$ tel que pour tout $y \in V_x$, $X(y)$ est dans $E_x^+$.
\end{definition}

\noindent Si un tel choix existe, on dit que $(\M,g)$ est {\it orientable en temps}. Dans la suite, on suppose $(\M,g)$ est {\it orient\'ee en temps}, c'est-\`a-dire que $(\M,g)$ est orientable en temps et qu'une orientation en temps continue a \'et\'e fix\'ee.

\begin{definition} 
\noindent $\,$

\begin{itemize}
\item On appelle {\it observateur} une courbe de genre temps. 
\item On appelle {\it observateur en un point $x\in \M$}  la donn\'ee d'un vecteur $\overrightarrow{v} \in T_x\M$ de genre temps, unitaire (i.e. $g(\overrightarrow{v},\overrightarrow{v})= 1$) et orient\'e positivement (i.e. $\overrightarrow{v} \in E_x^+$). 
\item  Soit $D$ un observateur (ou un observateur en un point). L'{\it espace global vu par $D$} en $x \in D$ est la partie de $\M$ qui est $g_x$-orthogonale, c'est-\`a-dire la partie $\E_x$ de $\M$ form\'ee de la r\'eunion de toutes les g\'eod\'esiques issues de $x$ et orthogonale \`a $D$ en $x$.
\item Soit $D$ un observateur param\'etr\'e par $c:I \to \M$. Le temps propre entre $A=c(a) \in \M$ et $B=c(b) \in M$ est donn\'e par 
$$\tau_{AB}=\int_a^b \sqrt{-g(c'(t),c'(t))} dt.$$  
\end{itemize}
\end{definition}

\begin{remark}
\noindent \,

\begin{enumerate}
\item De m\^eme qu'en relativit\'e restreinte et en m\'ecanique classique, si $D$ est un observateur, il existe une {\it param\'etrisation normale positive} de $D$, unique \`a translation en temps pr\`es, i.e. une param\'etrisation $c:I \to \M$ telle que pour tout $t$, $g(c'(t),c'(t)=1$ et $c'(t) \in E_{c(t)}^+$.
\item En relativit\'e restreinte, on n'avait pas besoin de la notion d'observateur en un point, bien que beaucoup de notions auraient pu se restreindre \`a cette d\'efinition (par exemple la vitesse ne d\'ependait que de la position et du vecteur tangent). 
\item L'espace global vu par un observateur en $x$ est une sous-vari\'et\'e
  de type espace
  de dimension $3$ au voisinage de $x$ 
 (m\^eme argument qu'en relativit\'e restreinte).
\end{enumerate}
\end{remark}

\noindent Avec ces d\'efinitions, parler de vitesse n'a pas vraiment de
sens. En effet, soit $D$  un observateur en un point $x \in \M$ 
dirig\'e par $\overrightarrow{v} \in T_x \M$ unitaire. On a besoin de
d\'ecomposer un vecteur d'un autre espace tangent $T_y\M$ en une composante
sur $\overrightarrow{v}$ et une composante sur
$[\overrightarrow{v}]^\perp$. Il y a plusieurs mani\`eres de la faire, mais aucune n'est canonique. 

\chapter{De la mati\`ere dans l'espace-temps} \label{matiere}
Ce chapitre a pour but d'arriver jusqu'\`a l'axiomatique de la relativit\'e g\'en\'erale pour  d\'ecrire le comportement de la mati\`ere. Avant d'en arriver \`a ce stade, il faut comprendre quels sont les probl\`emes pos\'es par la m\'ecanique classique et la relativit\'e restreinte. Quelle que soit la mani\`ere dont on construit la th\'eorie, il faut garder \`a l'esprit qu'un "observateur humain" doit percevoir les mouvements pr\'edits par les lois de Newton. Ces r\`egles ne peuvent en aucun cas \^etre remises en cause \`a vitesse faible (par rapport \`a celle de la lumi\`ere). La principale diff\'erence entre la relativit\'e g\'en\'erale et la m\'ecanique classique doit surtout se faire sentir soit \`a grande \'echelle, soit lorsque des vitesses importantes sont en jeu (par exemple, un GPS qui analyse tr\`es pr\'ecis\'ement la position d'un utilisateur \`a partir d'ondes tient compte des effets relativistes). C'est pourquoi dans ce chapitre, nous commen\c {c}ons par rappeler les lois utilis\'ees pour d\'ecrire le comportement de la mati\`ere en m\'ecanique classique et en  relativit\'e restreinte, ce qui nous am\`enera naturellement \`a l'axiomatique de la relativit\'e g\'en\'erale.  

\section{Particules et fluides} 

\noindent En m\'ecanique classique, relativit\'e restreinte et relativit\'e g\'en\'erale, la mati\`ere est suppos\'ee se composer de particules qui se d\'efinissent de la mani\`ere suivante.

\begin{definition}
En m\'ecanique classique, relativit\'e restreinte et relativit\'e g\'en\'erale, une particule est un couple $(\C,m)$ o\`u $\C$ est une courbe de genre temps et $m$ est un nombre positif ou nul, la {\it masse} de $p$.
\end{definition}
Autrement dit, une particule est un observateur muni d'une masse. Lorsqu'on prend en compte les ph\'enom\`enes \'electromagn\'etiques, on lui attribue \'egalement
une charge $e$.\\

\noindent Malheureusement, si on s'int\'eresse au mouvement de chaque
particule, les \'equations qui apparaissent m\^eme en m\'ecanique classique
sont quasiment irr\'esolubles. Cela conduit \`a consid\'erer la mati\`ere
comme un fluide.

\begin{definition} 
\noindent \,

\begin{enumerate}
\item Une {\it congruence de courbes} (terminologie de S. Hawking) 
sur un domaine $\Om$ de $M$ est une famille de courbes de type temps qui ne
se coupent pas. Plus pr\'ecis\'ement, il s'agit d'une famille de courbes 
$(\C_j)_{j\in X}$ de type temps telle que tout $x \in \Om$ admet un
voisinage ouvert $w_x$  et un diff\'eomorphisme (avec r\'egularit\'e
suffisante pour que tout soit bien d\'efini) 
$\phi_x:w_x \to I \times B$ (o\`u $I$ est un intervalle ouvert de $\mR$ et o\`u $B$ est une
boule ouverte de $\mR^3$) tel que pour tout $j \in X$, il existe un unique
$y \in B$ avec  $\phi_x(\C_j \cap w_x) = I \times \{ y \}$. 
\item Un {\it fluide} dans $\M$  est un couple $(\hbox{congruence de courbes}, \, \rho)$
  o\`u $\rho : \M \to \mR^+$ est une fonction appel\'ee {\it densit\'e de
    masse} du fluide. Ce couple devra v\'erifier une propri\'et\'e suppl\'ementaire  que l'on d\'efinira 
  plus tard (voir la "propri\'et\'e requise" ci-dessous). 
\end{enumerate}
\end{definition}
 

\noindent Physiquement, les courbes repr\'esentent les trajectoires de
chaque point du fluide. Avec ce point de vue, on ne voit plus les particules une \`a une.   Prenons un observateur $D$  
{\it attach\'e au fluide} (i.e. l'une des
courbes de la congruence). La   fonction
densit\'e de masse, physiquement, se d\'efinit comme suit: au point $x$,
consid\'erons {\it l'observateur $D_x$ fix\'e au fluide} (i.e. la courbe
qui passe par $x$). Le densit\'e de masse
est la limite du quotient de la masse mesur\'ee par $D_x$ (i.e. la somme
des masses des particules) contenue dans un voisinage $v_x$ de $\E_x$ (espace
vu
par $D$ en $x$) par le volume pour la m\'etrique riemannienne induite de
$v_x$ lorsque $v_x$ se r\'eduit autour de $x$. 

\begin{remark}
On pourrait penser que la fonction densit\'e de masse d\'ecrit compl\`etement le fluide \`a elle seule puisqu'elle indique la quantit\'e de mati\`ere pr\'esente \`a tout instant et \`a tout endroit. En fait, elle n'est pas suffisante : par exemple, sans la donn\'ee de la congruence de courbes, on n'a aucun moyen de d\'etecter la rotation d'une
particule sph\'erique.
\end{remark}

\begin{definition} 
Consid\'erons un fluide $F$ dans $\M$. Le {\it champ de vecteurs
    unitaire 
  associ\'e \`a $F$} est le champ de vecteurs $\overrightarrow{u}$  form\'e des
vecteurs tangents aux courbes du fluides, normaux
($T(\overrightarrow{u}) = 1$ en m\'ecanique classique, $T(\overrightarrow{u})= g(\overrightarrow{u},\overrightarrow{u})=-1$ en relativit\'e restreinte et $g(\overrightarrow{u},\overrightarrow{u})=-1$ en relativit\'e g\'en\'erale) et orient\'es positivement.
\end{definition} 

\noindent 
\begin{definition} \label{masseaurepos}
Consid\'erons une hypersurface $S$ de type espace
et  un fluide $F$ de densit\'e de masse $\rho$. 
On appelle {\it masse au repos
    de $F$ sur $S$} le flux du champ $\overrightarrow{u}$ \`a travers $S$.
En relativit\'e restreinte et relativit\'e g\'en\'erale, elle est d\'efinie par l'int\'egrale
$$ -\int_S  \rho g(\overrightarrow{u},\overrightarrow{n})$$
o\`u $\overrightarrow{u}$ est le champ de vecteurs unitaire, o\`u $\overrightarrow{n}$ est le
champ de vecteur $g$-orthogonal \`a $S$, unitaire ($g(\overrightarrow{n},\overrightarrow{n})=-1$) et
orient\'e positivement. 
\end{definition} 


\noindent Donnons quelques explications sur la ``masse au repos''. Supposons que $S$ soit $g$-orthogonale \`a $\overrightarrow{u}$, c'est-\`a
-dire, d'un point de vue physique, au "mouvement" du fluide. Alors, $g(\overrightarrow{u},\overrightarrow{n}) = -1$ (car $\overrightarrow{u} = \overrightarrow{n}$) et la masse au repos
est la masse de fluide que contient $S$ mesur\'ee par un observateur $D$ fix\'e
au fluide. \\

\noindent {\bf Propri\'et\'e requise (en m\'ecanique classique, relativit\'e restreinte et relativit\'e g\'en\'erale) :} 
 
\noindent {\it Soit $F$ un fluide, $\rho$ et  $\overrightarrow{u}$ respectivement  la densit\'e de masse
et le champ de vecteurs unitaire associ\'es \`a $F$. On impose que 
\begin{eqnarray} \label{div=0}
 \div(\rho
\overrightarrow{u}) = 0.
\end{eqnarray}}

\noindent Cette condition traduit le fait qu'il n'y a pas de perte de
mati\`ere entre deux instants donn\'es. Essayons de comprendre pourquoi. On se place dans le cadre de la relativit\'e g\'en\'erale ou restreinte. 
Reprenons la d\'efinition de la congruence de courbes: pour $x \in \M$ il
existe un voisinage $w_x$ diff\'eomorphe via $\phi_x$ \`a $I \times
B$. Pour simplifier, supposons que $I = ]0,1[$ et identifions $w_x$ \`a $]0,1[ \times
B$ (on confond $w_x$ et son image). Le bord de $w_x$ est
form\'e de trois parties: $S_0:= \{0\} \times B$, $S_1:=\{1\} \times B$  et
$S_2 := ]0,1[ \times
S^1$. On se place dans la situation la plus claire physiquement : 
 $S_0$ et $S_1$ sont
$g$-orthogonales au fluide - c'est-\`a-dire que les points de $S_0$ et $S_1$
sont tous simultan\'es pour un observateur fix\'e au fluide- alors que
$S_2$ est tangente au fluide. Notons  $\overrightarrow{n}_j$ le vecteur normal \`a 
$S_j$ ($i,j \not=2$ sinon $S_j$ n'est pas de genre espace). Notons aussi
$ds_g$
l'\'el\'ement de volume induit par $g$ sur $S_j$.

\noindent Avec le th\'eor\`eme de Stokes,
 
$$0= \int_{w_x} \div(\rho \overrightarrow{u}) dv_g = \sum_{i=1}^3 
\int_{S_j} \rho g(\overrightarrow{u},\tilde{\overrightarrow{n}}_j) ds_g.$$
o\`u $\tilde{\overrightarrow{n}}_j$ est le vecteur $g$-orthogonal \`a $S_j$, unitaire et
sortant. Autrement dit, $\tilde{\overrightarrow{n}}_0= -\overrightarrow{n}_0$ et
$\tilde{\overrightarrow{n}}_1= \overrightarrow{n}_1$. De plus, il est clair 
que 
$$\int_{S_2} \rho g(\overrightarrow{u},\tilde{\overrightarrow{n}}_2) ds_g = 0.$$
On obtient ainsi 
$$0=\int_{S_1} \rho g(\overrightarrow{u},\overrightarrow{n}_1) ds_g - \int_{S_0} \rho
g(\overrightarrow{u},\overrightarrow{n}_0) ds_g$$
ce qui montre que la masse au repos du fluide sur $S_0$ est la m\^eme que
sur $S_1$.\\

\noindent \`A l'\'echelle de l'univers, les fluides sont compos\'es
d'\'etoiles, de galaxies qui jouent le r\^ole de particules. Ces particules
s'entrochoquent rarement et les forces qui s'exercent entre elles ne sont
pas de nature \'electromagn\'etique (en fait, on verra qu'en relativit\'e g\'en\'erale les particules n'interagissent pas entre elles).  Cela conduit \`a introduire la d\'efinition suivante : 

\begin{definition}
Un {\it fluide parfait sans pression} est un fluide dans lequel les particules sont
ind\'ependantes les unes des autres et dans lequel il n'y a pas d'autre
\'energie que celle des particules (pas de chocs, c'est-\`a-dire pas de viscosit\'e,  pas de rotation)
\end{definition} 

\noindent Cette d\'efinition est parfaitement adapt\'ee au mod\`ele de la relativit\'e g\'en\'erale, o\`u comme on le verra plus tard, les particules sont suppos\'ees ne pas avoir d'interactions entre elles. Un {\it fluide parfait} inclut normalement la pression (qui est une \'energie suppl\'ementaire) mais  lorsque les particules du  fluide sont constitu\'ees d'\'etoiles, de plan\`etes et de galaxie, la pression est suppos\'ee nulle, sauf \`a l'int\'erieur des particules. Ce mod\`ele est aussi utilis\'e dans d'autres cadres, par exemple dans le cas d'un fluide de tr\`es faible viscosit\'e (par exemple en a\'erodynamique).

\section{En m\'ecanique classique} 
En m\'ecanique classique, les principes utilis\'es sont ceux de Newton, qui traduisent l'{\it attraction universelle} (deux objets quelconques s'attirent mutuellement), id\'ee qui sera compl\`etement abandonn\'ee en relativit\'e g\'en\'erale.   
Il y a deux points de vue diff\'erents pour mod\'eliser l'attraction universelle. Soit on utilise les lois de Newton, soit on utilise la notion de Lagrangien. Bien \'evidemment, quel que soit le point de vue choisi, on retrouve les m\^emes r\'esultats.  
\subsection{Point de vue du potentiel pour des particules}

\noindent {\it Ce paragraphe a pour but de formuler la loi de Newton qui
  traduit l'attraction universelle. } \\

\noindent Soit $p=(\C,m)$ une particule. 

\begin{definition} 
Le {\it potentiel} cr\'e\'e par $p$ est la fonction d\'efinie sur $\M$ par 
$$f_p(M) = - \frac{km}{d(p,M)}$$
o\`u $k$ est une constante universelle appel\'ee {\it constante de gravitation} et o\`u $d(M,p)$ est la distance introduite dans le paragraphe \ref{modelisation_mc} du chapitre \ref{modelisation}. 
\end{definition}

\noindent Le comportement de la mati\`ere est alors r\'egi par la \\

\noindent {\bf Loi de Newton} {\it Soient $p_1=(\C_1,m_1),\cdots,p_n=(\C_n,m_n)$ des particules. On suppose que les courbes $\C_i$ ne se coupent pas. Alors pour tout $i \in \{1,\cdots,n \}$ l'acc\'el\'eration de la particule $p_i$ en $M \in \C_i$ est 
$$\overrightarrow{a_i}(M) = - \sum_{j \not= i} \nabla_{x} f_j(M)$$
o\`u $\nabla_x$ est le gradient calcul\'e dans la direction de $\ker(T)$. } \\

\noindent L'acc\'el\'eration est calcul\'ee relativement \`a un observateur
galil\'een mais rappelons qu'elle ne d\'epend pas de l'observateur
galil\'een choisi. Ce syst\`eme diff\'erentiel d'ordre 2
est presque impossible \`a r\'esoudre d\`es qu'il y a $3$ particules ou plus en jeu ({\it probl\`eme des 3 corps}).\\

\subsection{Point de vue du potentiel pour un fluide parfait sans pression}

\noindent {\it Ce paragraphe sert  \`a faire deviner quelles seront les bons axiomes \`a poser en relativit\'e g\'en\'erale pour qu'\`a vitesse faible (par rapport \`a celle de la lumi\`ere), on puisse retrouver des lois proches de celles de Newton.}\\

\noindent  On rappelle que le potentiel cr\'e\'e par des particules $p_1=(\C_1,m_1),\cdots, p_1=(\C_1,m_1)$ est
$$f(M)= - k \sum_{i=1}^m \frac{m_i}{d(M,\C_i)}$$
o\`u $d(M,\C_i)$ est la distance de $M$ \`a l'unique point de $\C_i$ qui est simultan\'e \`a $M$. Par extension, si $F$ est un fluide parfait sans pression de densit\'e de masse $\rho$, on d\'efinit le potentiel cr\'e\'e par $F$ en posant  
$$f(M) = - \int_{M+\ker(T)} \frac{\rho(y)}{d(M,y)} dv_g(y)$$ 
o\`u rappelons-le $g$ est le produit scalaire dont nous avons muni $\ker(T)$ et o\`u l'on a choisi l'unit\'e de masse pour que $k = 1$. Soit maintenant $D$ un observateur galil\'een dirig\'e par $i_D$ unitaire, orient\'e positivement. 
On rappelle que d\`es lors qu'on choisit une origine $A\in D$ (ce qui correspond \`a choisir un instant $t=0$ pour $D$) $D$ "voit" $\M$ comme $\mR \times \ker(T)$ via l'isomophisme $\phi_D$ d\'ecrit dans le paragraphe \ref{modelisation_mc} du chapitre \ref{modelisation}. L'hyperplan $\{t\} \times \ker(T)$ correspond \`a l'espace observable par $D$ \`a l'instant $t$. Via cet isomorphisme, $f$ se r\'e\'ecrit 
$$f(t,x)= - \int_{\ker(T)} \frac{\rho(t,y)}{d(x,y)} dv_g(y)$$
pour tout $(t,x) \in \mR \times \ker(T)$. \\

\begin{remark}
Avec les m\^emes notations, si $\Om \subset \ker(T)$, l'int\'egrale
$$\int_\Om \rho(t,y) dv_g(y)$$
repr\'esente la masse du fluide qui se trouve dans $\Om$ (bien s\^ur, ce $\Om$ d\'epend de $D$ \`a l'instant $t$). 
\end{remark}

\noindent Maintenant, il faut se souvenir que la fonction de Green du laplacien  sur $(\ker(t),g)$ est 
$$G(x,y) = \frac{1}{4 \pi d(x,y)}$$
(on travaillera toujours avec le laplacien avec la convention de signe
suivante : il est \'egal \`a $- \sum_{i}^n \frac{\partial^2}{\partial x_i^2}$ lorsque la carte $(x_1,\cdots,x_n)$ est une isom\'etrie sur un ouvert de $\mR^n$ muni de sa m\'etrique standard) 
et ainsi 
$$f(t,x)= - 4 \pi \int_{\ker(T)} G(x,y) \rho(t,y) dv_g(y).$$   
Autrement dit, on a 

\begin{eqnarray} \label{Relation1}
\Delta_xf = - 4 \pi \rho.
\end{eqnarray}

\noindent La loi de Newton traduite sur les courbes du fluide est alors donn\'ee par 
\begin{eqnarray} \label{Relation2}
\overrightarrow{a}(t,x)= - \nabla_x f (t,x)
\end{eqnarray}

o\`u $\overrightarrow{a}(t,x)$ est l'acc\'el\'eration au point $x$ et au temps $t$ de la courbe du fluide passant par $(t,x)$. Maintenant, on rappelle que 
\begin{eqnarray} \label{Relation3}
\div(\rho \overrightarrow{u})=0
\end{eqnarray}
o\`u $\overrightarrow{u}$ est le champ de vecteurs unitaire associ\'e au fluide et  que cette relation traduit la conservation de masse. Lorsqu'on consid\'erait la mati\`ere particule par particule, cette relation \'etait juste remplac\'ee par le fait qu'il y avait \`a tout instant le m\^eme nombre de particules et que leur masse \'etait constante. \\

\noindent {\bf Les relations \eref{Relation1}, \eref{Relation2}
et \eref{Relation3} sont les relations qui r\'egissent le mouvement d'un fluide parfait sans pression en m\'ecanique classique. Ce sont elles que l'on va essayer de retrouver \`a vitesse faible en relativit\'e g\'en\'erale.}

\subsection{Point de vue du lagrangien}

\noindent {\it Ce paragraphe donne une formulation \'equivalente \`a la loi de Newton qui permet d'introduire naturellement les notions d'\'energie et d'impulsion qui seront \`a la base  de la th\'eorie en relativit\'e g\'en\'erale. Pour finir nous regarderons l'exemple d'une particule dans le vide. La lecture de ce paragraphe n'est pas indispensable pour comprendre d'o\`u vient l'\'equation d'Einstein. } \\

Le point de vue du  lagrangien consiste \`a voir les trajectoires des particules comme des chemins minimisant une fonctionnelle appel\'ee {\it fonctionnelle d'action}  (en quelque sorte des g\'eod\'esiques sauf que cette fonctionnelle d\'epend du syst\`eme physique). On travaille donc particule par particule. Physiquement, la fonctionnelle d'action calcule pour une trajectoire donn\'ee 
l'\'energie cin\'etique de la trajectoire moins l'\'energie potentielle cr\'e\'ee par les autres particules. En effet, les particules vont avoir tendance \`a suivre les trajectoires qui leur font d\'epenser le moins d'\'energie (\'energie cin\'etique) et qui va utiliser au maximum l'\'energie potentielle des autres particules. En fait, on va oublier cette interpr\'etation physique en relativit\'e restreinte.

\begin{definition}
\noindent \,

\begin{enumerate}
\item Un {\it lagrangien} d'une particule $p=(\C,m)$ dans un syst\`eme
  physique (i.e. dans un ensemble de particules contenant $p$) est une
  application $C^1$ $L: \M \times E  \to \mR$ qui v\'erifie plusieurs axiomes que nous pr\'eciserons plus tard et qui permettront de mod\'eliser les trajectoires des particules.
\item Soit $L$ un lagrangien de $p$ et $A,B \in \C$. Dans la suite on consid\'erera toujours que $B$ est ult\'erieur \`a $A$.   Notons $\tau_{AB}$ le temps entre $A$ et $B$. Soit $\C'$ une autre courbe de genre temps passant par $A$ et $B$ param\'etr\'ee par $\beta:[0,\tau_{AB}] \to \M$,  normale orient\'ee positivement et telle que $\beta(0)=A$ et $\beta(\tau_{AB})= B$ (on dira que $\beta$ est {\it admissible}). On d\'efinit la {\it fonctionnelle d'action} associ\'ee \`a $L$ entre $A$ et $B$ par 
$$S_{AB}(\C') \, (\hbox{ou } S_{AB}(\beta)) \, = \int_0^{\tau_{AB}} L(\beta(t),\beta'(t)) dt.$$    
\end{enumerate}
\end{definition}

\noindent Comme expliqu\'e plus haut, on veut que \\

\noindent \underline{Axiome 1} {\it
$$S_{AB} (p):= S_{AB}(\C) \leq S_{AB}(\C')$$
pour toute courbe $\C'$ de genre temps passant par $A$ et $B$ ou de mani\`ere \'equivalente 
$$S_{AB}(p) \leq S_{AB}(\beta)$$
pour tout param\'etrisation $\beta$ admissible.} \\

\noindent Le lagrangien d'une particule $p$ v\'erifiant l'axiome 1 n'est pas unique. Il est d\'efini \`a une diff\'erentielle totale pr\`es. Rappelons qu'une diff\'erentielle totale est une fonction $F: \M \times E \to \mR$ de la forme $F(x,\overrightarrow{v}) = df_x(\overrightarrow{v})$
 o\`u $f:\M \to \mR$.
 En effet, soit $F$ une telle fonction. Notons $L'= L +F$. 
Alors, puisque 
$$\int_0^{\tau_{AB}} F(\beta(t),\beta'(t)) dt = f(B) -f(A)$$ 
pour toute param\'etrisation $\beta$ admissible, les fonctionnelles d'action associ\'ees \`a $L$ et 
$L'$ ne diff\'erent que de la constante $f(B)-f(A)$ et l'axiome 1 est vrai
pour $L$ si et seulement si il est vrai pour $L'$.  Inversement, on a :

\begin{prop} Si  les
fonctionnelles d'action de deux lagrangiens $L$ et $L'$ diff\'erent d'une constante pour tous $A,B$,
alors $L$ et $L'$ diff\`erent d'une diff\'erentielle totale.
\end{prop}

\begin{proof} On d\'efinit la forme diff\'erentielle  $w$ par $w(x)(\overrightarrow{v})  = L(x,\overrightarrow{v}) -L'(c,\overrightarrow{v})$. 
Par hypoth\`ese, l'int\'egrale de  $w(x)$
le long d'un chemin ne d\'epend que des extr\'emit\'es de $c$. On fixe $q \in \M$ et on d\'efinit pour $x \in \M$ la fonction $f(x) = \int_c w$ o\`u $c$ est un chemin quelconque joignant $q$ \`a $x$.
Fixons $x \in \M$ et prenons une base $(e_1,e_2,e_3,e_4)$. La forme $w$ s'\'ecrit $w= \sum_{i=1}^4 a_i dx_i$ o\`u les $dx_i$ sont les fonctions coordonn\'ees dans cette base.   Pour $t$ petit, remarquons que   
$f(x + t e_i ) = \int_c w + \int_{c_i} w$
o\`u $c_i(s) = x+se_i$ pour $s \in [0,t]$. Comme le premier terme ne d\'epend pas de $t$,  
$$\frac{d}{dt}_{t=0} f(x+te_i)= \frac{d}{dt}_{t=0} \int_O^t w(c(s)) (c'(s)) ds = \frac{d}{dt}_{t=0} \int_O^t a_i(c(s)) ds = a_i(x).$$
Cela montre que $w = df$ et que $L$ et $L'$ diff\`erent d'une diff\'erentielle totale. \\
\end{proof}

\noindent   Soit $D_A$ un observateur galil\'een
  muni d'une origine $A$ (voir paragraphe \ref{modelisation_mc} du chapitre \ref{modelisation}). On a vu que $D_A$ d\'eterminait de mani\`ere naturelle un isomorphisme entre $\M$ et $\mR \times \ker(T)$ en \'ecrivant, pour tout point $M \in \M$ 
  $$\overrightarrow{AM} = \overrightarrow{k} + a i_D$$
  o\`u $a \in \mR$, $\overrightarrow{k} \in \ker(T)$ et o\`u $i_D$ est le vecteur unitaire orient\'e positivement qui dirige $D$.   
  
\begin{definition} Soient $L$ un lagrangien d'une particule $p$ dans un syst\`eme physique. Le {\it lagrangien de $p$ vu par $D_A$ et associ\'e \`a $L$} est donn\'e par 
\[ L_{D_A} : \left| \begin{array}{ccc} 
\mR\times  \ker(T) \times \ker(T) & \to &\mR   \\
(t,x,\overrightarrow{v}) & \mapsto & L(M,\overrightarrow{v} + i_D) 
\end{array} \right. \]
o\`u $M$ est tel que $\overrightarrow{AM}= \overrightarrow{k} + t i_D$.
\end{definition}

\noindent Soit $\beta:[O,\tau_{AB}] \to \M$ une param\'etrisation admissible entre $A$ et un autre point $B$ de $D$. On peut lui associer 
\[ \tilde{\beta} \left| \begin{array}{ccc}
 [0,\tau_{AB}] & \to&  \mR \times \ker(T) \\
 t & \mapsto & (\alpha(t) , \overrightarrow{k}(t))
 \end{array} \right. \]
o\`u comme dans la d\'efinition ci-dessus, $\overrightarrow{A\beta(t)}= \overrightarrow{k} + \alpha(t) i_D$. Remarquons que comme $T(\beta'(t)) = 1$, on a $T(\alpha'(t) i_D) = 1$ et $\alpha'(t)= 1$. Comme de plus $\alpha(0) = 0$ (puisque $\beta(0)= A$), on a $\alpha(t) = t$.  
Ainsi $\tilde{\beta}(t)= (t,\overrightarrow{k})$. 
De cette mani\`ere, $\overrightarrow{k}'(t) = \overrightarrow{v}_{\beta/D}$
(voir le point 1 de la Remarque \ref{rem_vitesse}). On a aussi $\overrightarrow{k}'(t) = \beta'(t) - i_D$. On en d\'eduit que 
$$L_{D_A}(t,\overrightarrow{k}(t),\overrightarrow{k}'(t)) = L(\beta(t), \beta'(t)).$$
Cela justifie cette d\'efinition d'autant que si on pose,
$$\tilde{S}_{AB}(\tilde{\beta})= \int_0^{\tau_{AB}} L_{D_A}(t,\overrightarrow{k}(t),\overrightarrow{k}'(t)) dt$$
 la courbe $\tilde{c}$ associ\'ee \`a la courbe $c$ param\'etrant la courbe $\C$ de $p$ minimise la fonctionnelle 
 $\tilde{S}_{AB}$ parmi tous les $\tilde{\beta}$. \\

\noindent On a maintenant le r\'esultat suivant (voir Avez, calcul diff\'erentiel)

\begin{theorem} \label{eq_euler}
Soit $F: \mR\times \ker(T) \times \ker(T) \to \mR$ une fonction $C^2$. Supposons qu'une courbe $k: [t_1,t_2] \to \ker(T)$ minimise 
$$S(k) = \int_{t_1}^{t_2} F(t,k(t),k'(t)) dt$$
parmi toutes les courbes normales orient\'ees positivement,
 alors on a 
\begin{eqnarray} \label{euler_eq} 
d_xF_{(t,k(t),k'(t))} = d/dt \left(d_v F_{(t,k(t),k'(t))} \right)
\end{eqnarray}
 o\`u $d_x$ et $d_v$ repr\'esentent respectivement les diff\'erentielles partielles relativement aux deuxi\`eme et troisi\`eme variables. 
 \end{theorem}

\noindent Ce th\'eor\`eme calcule l'\'equation d'Euler d'un minimiseur de la fonctionnelle d'action et fournit ainsi une \'equation diff\'erentielle dont les solutions donnent les trajectoires des particules. \\

\noindent On va maintenant d\'efinir l'{\bf \'energie} et l'{\bf impulsion}
d'une particule dans un syst\`eme physique vu par un observateur $D_A$. Dans le premier chapitre, on explique bri\`eve--ment leur interpr\'etation physique. Par ailleurs, les lois physiques donn\'ees dans ce m\^eme chapitre impliquent que ces grandeurs doivent \^etre constantes avec le temps. D'un point de vue math\'ematique, ce sont des int\'egrales premi\`eres du syst\`eme. On gardera ce point de vue math\'ematique ici. On verra aussi que leur d\'efinition impose des conditions tr\`es restrictives mais qui seront remplies pour le cas d'une particule dans le vide. On constatera au final que les r\'esultats trouv\'es correspondent \`a ceux qui avaient \'et\'e obtenus par des intuitions physiques au Chapitre \ref{raisonnements}. 

\begin{prop} \label{energie_def}
Avec les m\^emes notations que ci-dessus, on suppose que $L_{D_A}$ ne d\'epend pas de la premi\`ere variable. Soit $c:[0,\tau_{AB}] \to \M$ l'unique courbe admissible qui param\`etre la courbe $\C$ de la particule. Notons $\tilde{c}:[0,\tau_{AB}] \to \mR \times \ker(T)$, $\tilde{c(t)} = (t,\overrightarrow{k}(t))$ la courbe associ\'ee vue par l'observateur $D_A$ (voir ci-dessus). Alors le nombre 
$$\E(t) = (d_v L_{D_A})_{(\overrightarrow{k}(t),\overrightarrow{k}'(t))}(\overrightarrow{k}'(t)) - L_{D_A} (\overrightarrow{k}(t),\overrightarrow{k}'(t))$$
est constant. On l'appelle l'{\bf \'energie de la particule vue par $D_A$}.
\end{prop}

\noindent Insistons encore une fois sur le fait que cette proposition-d\'efinition n'a de sens
que si $L_{D_A}$ ne d\'epend pas de la premi\`ere variable. Notons aussi
que si tel est le cas, il n'y a aucune raison que cette hypoth\`ese soit
vraie si on change d'observateur galil\'een. Pour comprendre ce qui se
passe physiquement, imaginons qu'un observateur \'etudie une particule dans
le vide. L'\'energie de cette particule est la somme de son \'energie au
repos (qui est suppos\'ee nulle en m\'ecanique classique) et de son
\'energie cin\'etique. Supposons que  cet observateur soit en mouvement
irr\'egulier par rapport \`a la particule. Imaginons par exemple qu'il soit
soumis \`a des forces \'electromagn\'etiques et que la particule soit
neutre \'electriquement. Alors, l'observateur va mesurer une \'energie pour
la particule qui est non constante dans le temps (elle va d\'ependre de la
vitesse de la particule par rapport \`a l'observateur).  Pour avoir une
bonne d\'efinition d'\'energie, il faut  que l'observateur soit d'une
certaine mani\`ere li\'e au syst\`eme.  \\

\begin{proof} 
On a en utilisant l'\'equation \eref{euler_eq}

\begin{eqnarray*} 
\begin{aligned} 
\frac{d}{dt} \E(t)  &  =   \left( (d_x L_{D_A})_{(\overrightarrow{k}(t),\overrightarrow{k}'(t))}(\overrightarrow{k}'(t)) \right) \\
&  + (d_v L_{D_A})_{(\overrightarrow{k}(t),\overrightarrow{k}'(t))}(\overrightarrow{k}''(t))  - \frac{d}{dt} \left( L_{D_A}(\overrightarrow{k}(t), \overrightarrow{k}'(t)) \right).
\end{aligned}
\end{eqnarray*}
Or 
$$\frac{d}{dt} \left( L_{D_A}(\overrightarrow{k}(t), \overrightarrow{k}'(t)) \right) = (d_x L_{D_A})_{(\overrightarrow{k}(t),\overrightarrow{k}'(t))}(\overrightarrow{k}'(t))  
 + (d_v L_{D_A})_{(\overrightarrow{k}(t),\overrightarrow{k}'(t))}(\overrightarrow{k}''(t)).$$
D'o\`u $\frac{d}{dt} \E(t)= 0$, ce qui prouve la proposition. 
\end{proof}

\noindent De m\^eme, on d\'efinit l'impulsion de la mani\`ere suivante:

\begin{prop} \label{impulsion_def}
On utilise les m\^emes notations que dans la proposition pr\'ec\'edente mais cette fois, on suppose que $L_{D_A}$ ne d\'epend pas des deux premi\`eres variables. Alors
$$\overrightarrow{P}= (d_v L_{D_A})_{\overrightarrow{k}'(t)} \in (\ker (T))^* \sim \ker(T)$$
est un vecteur constant que l'on appelle {\bf impulsion de la particule vue par $D_A$}
\end{prop}

\noindent Notons que dans l'\'enonc\'e ci-dessus, l'identification entre $\ker(T)$ et $\ker(T)^*$ est donn\'ee par le produit scalaire $g$ (voir Paragraphe \ref{modelisation_mc}). Encore une fois, les conditions extr\^emement restrictives d'application de la proposition (d\'ependance de $L_{D_A}$ de la troisi\`eme variable uniquement) seront v\'erifi\'ees dans le cas d'une particule dans le vide.  

\begin{remark} Les d\'efinitions ci-dessus ne sont pas tout \`a fait
  rigoureuses.
 En effet, on a vu qu'un lagrangien \'etait d\'efinie \`a une
 diff\'erentielle totale pr\`es. Si maintenant on remplace $L$ par $L+F$
 o\`u $F$ est une diff\'erentielle totale, 
on va trouver une nouvelle \'energie (et impulsion) qui seront les m\^emes
que celles trouv\'ees avec $L$ mais auxquelles on aura ajout\'e une
constante. On verra que pour obtenir un mod\`ele physique r\'ealiste, 
il faudra que le lagrangien d\'epende de la masse de la particule. Il y
aura alors un seul choix de constante possible pour que l'\'energie et l'impulsion d'une particule de masse nulle soient nulles. 
\end{remark}

\noindent {\bf Exemple d'une particule dans le vide.} 

\noindent Comme expliqu\'e ci-dessus, on pourra d\'efinir l'\'energie et l'impulsion de la particule pour tout observateur galil\'een. 

\noindent Dans ce cas pr\'ecis, on consid\`ere une particule $p=(\C,m)$ et on va se donner 

\noindent \underline{Axiome 2} {\it un lagrangien de $p$ est invariant par les isom\'etries de $\M$ i.e. pour toute isom\'etrie $\phi:\M \to \M$ et pour tous $A,B \in M$, il existe $c \in \mR$ tel que pour tout $\beta$ admissible, on ait $S_{AB}(\phi\circ \beta)= S_{AB}(\beta)$.} \\

\noindent Cet  axiome traduit le fait  que physiquement, il n'y a pas de direction privil\'egi\'ee dans l'univers et qu'une particule que l'on "bouge" par une isom\'etrie (position et vitesse) \`a un instant donn\'e a une trajectoire qui est "boug\'ee" de la m\^eme mani\`ere (i.e. par la m\^eme isom\'etrie). Alors on montre 

\begin{theorem} \label{part_vide}
En consid\'erant les  axiomes $1$ et $2$, il existe dans $\bar{L}$ (classe des lagrangiens d\'efinis \`a une diff\'erentielle totale pr\`es) un lagrangien $L_0$ tel que pour tout $\overrightarrow{v}$ unitaire ($T(\overrightarrow{v})=1$) et orient\'e positivement 
$$L_0(x,\overrightarrow{v}) = a \| \overrightarrow{v} -i_D \|^2$$
o\`u $i_D$ est le vecteur directeur unitaire orient\'e positivement d'un observateur galil\'een fix\'e $D$. 
\end{theorem}

\noindent La d\'emonstration de ce r\'esultat n'est pas \'evidente du tout et sera omise ici. 

\begin{remark}
\noindent \,

\begin{enumerate}
\item Dans l'\'enonc\'e ci-dessus, $\| \cdot\|$ est la norme associ\'ee \`a $g$ (voir Paragraphe \ref{modelisation_mc}). La d\'efinition a bien un sens car $\overrightarrow{v} - i_D \in \ker(T)$. 
\item La forme de $L_0$ n'est donn\'e que pour des vecteurs unitaires mais c'est \`a ces vecteurs que l'on applique $L_0$.  
\item Le th\'eor\`eme dit "il existe $D$ tel que ..." mais en fait, l'observateur $D$ peut \^etre choisi arbitrairement. En effet, si dans la d\'efinition de $L_0$, on remplace $i_{D}$ par $i_{D'}$ ($D'$ \'etant un autre observateur galil\'een), on obtient un lagrangien $L_0'$ qui diff\`ere de $L_0$ par une diff\'erentielle totale. 
\end{enumerate}
\end{remark}

\noindent Dans ce th\'eor\`eme, on peut a priori prendre $a=0$ i.e. $L =0$ et les axiomes $1$ et $2$ sont bien v\'erifi\'es mais toute trajectoire est alors minimisante ce qui ne correspond pas \`a la r\'ealit\'e physique. On va poser $a = \frac{m}{2}$. Fixons maintenant un observateur $D$. Pour simplifier, prenons celui que l'on a choisi dans le Th\'eor\`eme  \ref{part_vide}. On a alors par d\'efinition $L_{D_A}(t,x,\overrightarrow{v}) = \frac{m}{2} \| \overrightarrow{v} \|^2$ qui ne d\'epend ni de $t$ ni de $x$. L'\'energie de la particule est donn\'ee par (on conserve les notations utilis\'ees lorsqu'on a d\'efini l'\'energie)  
$$
\E = (d_vL_{D_A})_{\overrightarrow{k}'(t)}( \overrightarrow{k}'(t)) -   \frac{m}{2} \| \overrightarrow{k}'(t) \|^2$$
et puisque $(d_v \| \overrightarrow{v} \|^2)_{\overrightarrow{k}'(t)} (\overrightarrow{k}'(t) ) = 2 \| k'(t) \|^2$, on trouve 
$$\E= \frac{m}{2}  \| \overrightarrow{k}'(t) \|^2.$$
On remarque que $\overrightarrow{k}'(t)$ repr\'esente la vitesse de la particule par rapport \`a $D$. Ainsi, On trouve que $E$ est \'egale \`a l'\'energie cin\'etique de la particule au sens habituel ($1/2 m v^2$).
Comme on l'a expliqu\'e plus haut, si on prend un autre lagrangien dans la m\^eme classe, on va trouver la m\^eme valeur de l'\'energie plus une constante. On fixe cette constante \`a $0$ pour que l'\'energie d'une particule de masse nulle soit nulle. 

De la m\^eme mani\`ere on trouve que l'impulsion est donn\'ee par 
$$\overrightarrow{P}= m \overrightarrow{k}'(t).$$ 
La trajectoire de la particule minimise la fonctionnelle d'action. La Proposition \ref{impulsion_def} nous dit alors que le vecteur vitesse $\overrightarrow{k}'(t)$ de la particule par rapport \`a $D$ doit \^etre constant. Puisque $D$ est arbitraire, on en d\'eduit que la trajectoire de la particule est une droite. \\

\noindent Lorsque le syst\`eme physique consid\'er\'e est compos\'e de $n$ particules, l'axiome $2$ ne permet plus de conclure. En fait, on postulera directement la valeur du lagrangien d'une particule pour retrouver la loi de Newton:\\

\noindent \underline{Axiome 2'}
{\it  Le lagrangien d'une particule $p_1= (\C_1,m_1)$ dans un syst\`eme physique $p_i= (\C_i,m_i)$ ($i \in \{1,\cdots,n\}$) est donn\'e par 
 
 $$L_1(x,\overrightarrow{v}) = \frac{m_1}{2} \|v-i_D \|^2 - k \sum_{i=2}^n \frac{m_i}{d(M,\C_i)}$$
o\`u $i_D$ est le vecteur unitaire positivement orient\'e d'un observateur galil\'een fix\'e $D$, o\`u $k$ est la constante de gravitation et o\`u $d(M,\C_i)$ est la distance de $M$ au point $M_i$ de $\C_i$ qui est simultan\'e \`a $M$.} \\

\noindent On remarque que ce lagrangien est en gros l'\'energie cin\'etique de $p_1$ moins l'\'energie potentielle des autres particules.

\section{En relativit\'e restreinte} \label{matiere_rr}

\subsection{Point de vue du lagrangien pour une particule dans le vide}\label{rr_lagrangien} 

\noindent {\it Ce paragraphe a pour but de calculer l'\'energie et l'impulsion d'une particule dans le vide en relativit\'e restreinte. Ce sera un bon point de d\'epart pour la th\'eorie de la  relativit\'e g\'en\'erale du fait que, comme on le verra, on fera l'hypoth\`ese que les particules sont ind\'ependantes les unes des autres : chacune se comportera comme une particule dans le vide. La lecture de ce paragraphe n'est pas indispensable pour comprendre d'o\`u vient l'\'equation d'Einstein.  }

\noindent La loi de Newton pose de nombreux probl\`emes. Simplement par son
\'enonc\'e, il y a interaction entre particules et de mani\`ere
sous-jacente, il y a le probl\`eme de la simultan\'eit\'e.  Notamment (il
faut essayer pour s'en convaincre), la loi de Newton am\`ene \`a
consid\'erer des vitesses plus grandes que celle de la lumi\`ere : en effet, si un objet change de place, son potentiel newtonnien est trnasform\'e en cons\'equence et son influence sur l'univers tout entier est instantan\'ement modifi\'e. L'information a donc \'et\'e transmise avec une vitesse infinie. Pour une seule particule dans le vide, le
principe lagrangien ne s'appuie pas sur ces interactions entre particules et on peut regarder ce qui se passe en relativit\'e restreinte.  
Dans ce cadre, on cherche un lagrangien qui v\'erifie les axiomes $1$ et $2$. L'exemple le plus simple est clairement de poser $L=constante$. On a vu que ce choix, en m\'ecanique classique, m\^eme s'il ne contredisait pas les axiomes $1$ et $2$, n'avait aucune chance de mod\'eliser la r\'ealit\'e physique puisque toute trajectoire minimiserait alors la fonctionnelle d'action. Comme on va le voir, la situation est diff\'erente en relativit\'e restreinte. \\

\noindent Soit donc $p=(\C,m)$ une particule dans le vide. On va poser $L= - m$, choix que l'on justifiera plus tard. Comme en m\'ecanique classique, fixons $A,B \in \C$ et prenons $\C'$ une autre courbe de genre temps passant par $A$ et $B$. La fonctionnelle d'action associ\'ee est 
$$S_{AB}(\C') (\hbox{ ou } S_{AB}(\beta)) := \int_0^{\tau_{AB}} (-m) dt = -m \tau_{AB}$$ 
o\`u encore une fois les $\beta$ admissibles sont les   param\'etrisations des courbes  $\C'$  
d\'efinies sur $[0,\tau_{AB}]$, normales de type temps orient\'ees
positivement et telles que $\beta(0)=A$ et $\beta(\tau_{AB})=B$.  
On rappelle que le temps propre d'un point $c(a)$  \`a un point $c(b)$  associ\'e \`a une courbe $c$ est 
d\'efini par 
$$\tau_{AB}=\int_a^b \sqrt{-T(c'(t))} dt.$$

\noindent Soit maintenant un  $D$ un observateur galil\'een dirig\'e par $i_D$ vecteur unitaire orient\'e positivement et $\beta:[a,b] \to \M$ une param\'etrisation d'une courbe $\C'$  que l'on suppose normale pour l'observateur $D$ (et qui n'est donc pas admissible). 
On veut trouver l'expression du lagrangien de la particule relativement \`a l'observateur $D$. Pour cela, on \'ecrit
$$ \overrightarrow{A\beta(t)}= \overrightarrow{k}(t) + \alpha(t) i_D$$
o\`u $\overrightarrow{k}(t) \in [i_D]^\perp$, $\alpha(t) \in \mR$ et   
o\`u l'on a choisi $A$ comme origine. On verra par la suite que le r\'esultat obtenu ne d\'epend pas du choix de l'origine.  
On obtient ainsi une courbe associ\'ee dans $\mR \times [i_D]^\perp$ $\tilde{\beta}(t):=(\alpha(t), \overrightarrow{k}(t))$. Comme en m\'ecanique classique, puisque $\beta$ est normale par rapport \`a l'observateur $D$, on a $\alpha'(t) = 1$ et $\overrightarrow{k}'(t)$  est la vitesse de $\beta$ par rapport \`a $D$. 
On cherche le lagrangien $L_D$ de la particule vue par $D$, c'est-\`a-dire $L_D :\mR \times [i_D]^\perp \times [i_D]^\perp   \to \mR$. Plus pr\'ecis\'ement, on cherche une fonction $L_D$ telle que la fonctionnelle d'action associ\'ee \`a $\beta$ (que l'on reparam\`etre pour qu'elle soit admissible) soit \'egale \`a celle de $\tilde{\beta}$  et ce, pour tout $\beta$.  
 Observons que
comme $T(i_D)=-1$, on a:

$$T(\beta'(t)) = T(\overrightarrow{k}'(t)) -1 =v_t^2-1$$
o\`u $v_t:=\sqrt{T(\overrightarrow{k}'(t))}$ est la norme de la vitesse de la particule d\'ecrite par la courbe  $\beta(t)$ par rapport \`a $D$. 
 Ainsi
$$S_{AB}(\C')= -m \int_a^b \sqrt{1-v_t^2} dt=-m \tau_{AB}$$
o\`u $A= \beta(a)$ et $B=\beta(B)$.
Autrement dit, le {\it lagrangien de la particule vue par $D$} est 

\[ L_D \left| \begin{array}{ccc}
\mR \times [i_D]^\perp \times [i_D]^\perp  & \to&  \mR \\
(t,x,\overrightarrow{v}) & \mapsto & -m\sqrt{1 -T(\overrightarrow{v})} 
\end{array} \right. \]

\noindent La fonctionnelle d'action associ\'ee redonne bien $S_{AB}(\C')$.

\begin{remark} 
Si le vecteur vitesse $v$ est petit par rapport \`a la vitesse de la
lumi\`ere (qui rappelons, avec nos conventions, vaut $1$), on voit que 
$$L_D(\overrightarrow{v}) \sim -m + \frac{1}{2} mv^2$$
($v= T(\overrightarrow{v})$). Comme $-m+\frac{1}{2} mv^2$ a une fonctionnelle d'action dont les minimiseurs sont les m\^emes que celle de  
$\frac{1}{2} mv^2$, qui \'etait le lagrangien de $p$ vue par $D$ en
m\'ecanique classique, on remarque que pour des vitesses petites, les
trajectoires de $p$ v\'erifient les m\^emes principes qu'en m\'ecanique
classique. C'est la premi\`ere raison pour laquelle on a choisi de prendre $L= -m$. On en verra une deuxi\`eme apr\`es le calcul de l'\'energie. 
\end{remark}
 
\noindent Revenons maintenant au  Th\'eor\`eme \ref{eq_euler} et aux  Propositions
\ref{energie_def} et \ref{impulsion_def}. Leurs preuves ne font intervenir
que la structure d'espace affine et pas la signature de la forme
quadratique $T$. Autrement dit, ils restent valables en relativit\'e
restreinte. Comme $L_D$ ne d\'epend que de la troisi\`eme variable, on peut
d\'efinir l'\'energie et l'impulsion. Fixons $\beta=c$ o\`u  $c:[a,b] \to
\M$ est une
param\'etrisation normale pour $D$ de la courbe $\C$ associ\'ee \`a la
particule $p$ consid\'er\'ee. En gardant les m\^emes notations que
ci-dessus, on a   pour l'{\it \'energie de $p$ vue par $D$} :

$$\E(t) = (d L_D)_{\overrightarrow{k}'(t)} (\overrightarrow{k}'(t)) - L_D(\overrightarrow{k}'(t)).$$
Puisque $L_D (\overrightarrow{v}) = -m\sqrt{1 -T(\overrightarrow{v})}$, on a 
\begin{eqnarray*}
\begin{aligned}
\E(t) &  = \frac{m T(\overrightarrow{k}'(t))}{\sqrt{1-T(\overrightarrow{k}'(t))}} + m
  \sqrt{1-T(\overrightarrow{k}'(t))} \\
& = \frac{m}{\sqrt{1-v_t^2}}.
\end{aligned}
\end{eqnarray*}

\begin{remark}
 Pour une particule au repos (c'est-\`a-dire vue par un observateur pour lequel $v=0$), on retrouve que l'\'energie de la particule est 
{\bf  $E= mc^2$} (voir Chapitre \ref{raisonnements}) o\`u $c$ est la vitesse de la lumi\`ere (qui dans nos unit\'es vaut $1$). Cela fournit une deuxi\`eme raison de d\'efinir le lagrangien par $L= -m$ pour une particule de masse $m$. Une diff\'erence majeure avec la m\'ecanique classique  est qu'en relativit\'e restreinte une particule au repos a une \'energie non nulle. 
\end{remark}

\noindent De m\^eme, on calcule que l'{\it impulsion de $p$ vue par $D$}
est donn\'ee par 
$$\overrightarrow{P} = \frac{m \overrightarrow{k}'(t)}{\sqrt{1-v_t^2}}.$$
On retrouve la formule trouv\'ee par intuition physique dans le Chapitre 1. Comme en m\'ecanique classique, on d\'eduit des Propositions
\ref{energie_def} et \ref{impulsion_def} que la trajectoires de $p$ est une
droite.

\begin{remark} \label{maxim}
Puisque $S_{AB}(\beta) =-m \tau_{AB}$, une particule dans le vide a une
courbe qui maximise $\tau_{AB}$.  
\end{remark}
\subsection{\'Energie d'un fluide parfait sans pression} \label{efpsp}  

\noindent {\it Ce paragraphe a pour but de d\'efinir l'\'energie d'un fluide parfait sans pression en relativit\'e restreinte. Cela permettra de donner en relativit\'e g\'en\'erale une d\'efinition naturelle du tenseur d'\'energie-impulsion. La lecture de ce paragraphe n'est pas indispensable pour comprendre d'o\`u vient l'\'equation d'Einstein. }

\noindent   En relativit\'e restreinte, 
$\M$ est un espace affine muni d'une forme bilin\'eaire sym\'etri--que $g$ (ou d'une forme quadratique  $T$) de
signature $(-,+,+,+)$. C'est donc un cas particulier de vari\'et\'e lorentzienne dont la m\'etrique
est en tout point \'egale \`a $g$. 
Soit $F$ un fluide parfait sans pression de densit\'e de masse $\rho$.

\begin{definition}
\noindent Soit $D$ un observateur galil\'een dirig\'e par $i_D$ unitaire et
orient\'e positivement. 

\begin{enumerate}
\item La {\it densit\'e d'\'energie du fluide $F$ par rapport \`a $D$} est la
  fonction d\'efinie pour $x \in \M$ par  
$$e_D(x) := \frac{\rho}{1-v_x^2}$$
o\`u $v_x^2$ est la vitesse (en norme) du fluide en $x$ pour l'observateur
$D$ (i.e. la 
vitesse de la courbe qui passe par $x$ relativement \`a $D$).
\item Soit $S$ un voisinage d'un point $x \in \M$, $S \subset x + [i_D]^\perp$
  (rappelons que $x+ [i_D]^\perp$ est l'ensemble des points simultan\'es
  \`a $x$ pour $D$). L'{\it \'energie du fluide \`a travers $S$} est
  l'int\'egrale
$$\E_S=\int_S e_D ds_g.$$
\end{enumerate} 
\end{definition}

\noindent Justifions cette d\'efinition. Supposons que sur $S$ la vitesse
du fluide par rapport \`a $D$ est constante (en norme) \'egale \`a $v$. Alors par d\'efinition, 
\begin{eqnarray} \label{energie_cst} 
\E_S= \frac{1}{1-v^2} \int_S \rho ds_g.
\end{eqnarray}   
Par ailleurs, $v = \sqrt{ g(\overrightarrow{v}_{F/D}, \overrightarrow{v}_{F/D})}$ et $\overrightarrow{v}_{F/D}(x)= \frac{\overrightarrow{k}}{a}$ o\`u l'on a \'ecrit comme
ci-dessus
\begin{eqnarray} \label{vit_fluide}
 \overrightarrow{u}(x) = \overrightarrow{k} + a i_D.
\end{eqnarray}
Ici $\overrightarrow{k} \in [i_D]^\perp$, $a \in  \mR$ et $\overrightarrow{u}$ est le champ
de vecteurs unitaire associ\'e \`a $F$. 
Maintenant, notons $m$ la masse au repos du fluide sur $S$. On a par d\'efinition
\begin{eqnarray} \label{masse_au_repos}
m= -\int_S \rho g(\overrightarrow{u},i_D) ds_g
\end{eqnarray}
car $i_D$ est, puisque $S \subset x + [i_D]^\perp$, le champ de vecteurs $g$-orthogonal \`a $S$ unitaire et positivement orient\'e. 
D'apr\`es l'\'equation \eref{vit_fluide}, on a en utilisant le fait que $\overrightarrow{k} \perp i_D$,
\begin{eqnarray} \label{ixid}
g(\overrightarrow{u},i_D)=  -a. 
\end{eqnarray}  
Encore une fois avec \eref{vit_fluide}, 
$$-1 =g(\overrightarrow{u},\overrightarrow{u}) = g(\overrightarrow{k},\overrightarrow{k}) - a^2$$
et comme 
$$g(\overrightarrow{k},\overrightarrow{k}) = a^2 g(\overrightarrow{v}_{F/D}, \overrightarrow{v}_{F/D}) = a^2 v^2,$$
cela donne 
$$-1= a^2(v^2 -1).$$
Puisque $a>0$ (car $\overrightarrow{u}$ et $i_D$ sont positivement orient\'es), on obtient 
$a = \frac{1}{\sqrt{1-v^2}}$. En injectant cette valeur dans \eref{ixid}, on 
obtient que 
\begin{eqnarray} \label{gii}
g(\overrightarrow{u},i_D)= - \frac{1}{ \sqrt{1-v^2}}.
\end{eqnarray}  
Avec \eref{masse_au_repos}, cela donne  
$$m = \frac{1}{\sqrt{1-v^2}} \int_D \rho ds_g.$$
Avec \eref{energie_cst}, on obtient une \'energie
$$\E_S = \frac{m}{\sqrt{1-v^2}}.$$
On retrouve en particulier la valeur de l'\'energie pour une particule dans le vide vue par l'observateur $D$ (voir le paragraphe \ref{matiere_rr}). \\

\section{En relativit\'e g\'en\'erale}

\noindent {\bf Diff\'erence fondamentale avec la m\'ecanique classique}
En m\'ecanique classique, on consid\'erait que les particules s'attiraient
entre elles. En relativit\'e, le comportement de la mati\`ere est r\'egi
par deux axiomes. \\

\noindent \underline{Axiome 1 :} {\it Les courbes des particules, param\'etr\'ees
par leur temps propres (c'est-\`a-dire que la param\'etrisation est normale
positive), sont des g\'eod\'esiques de type temps de $(\M,g)$.} \\

\noindent On consid\`ere que le tenseur $g$ (et donc ses g\'eod\'esiques)
contient toutes les informations sur la mati\`ere. \\

\noindent \underline{Axiome 2 :} {\it Cet axiome donne pr\'ecis\'ement le lien
entre le tenseur $g$ et la mati\`ere. Il sera pr\'ecis\'e plus tard.}\\

\noindent Toute la difficult\'e revient justement \`a trouver ce deuxi\`eme
axiome de mani\`ere \`a ce que, \`a vitesse faible, on retrouve les lois de
la m\'ecanique classique. On veut \'eviter de consid\'erer la mati\`ere
particule par particule : on a vu que cela conduit \`a des \'equations affreusement
compliqu\'ees \`a r\'esoudre (c'est ce qui se passe en m\'ecanique
classique d\'es qu'il y a trois particules ou plus). On gardera donc le point de vue des fluides. 

\subsection{Fluides parfaits sans pression et tenseur d'\'energie-impulsion} 

\noindent Puisqu'en relativit\'e g\'en\'erale, on consid\`ere que les particules sont ind\'ependantes les unes des autres et que leur trajectoire ne d\'epend
que de la m\'etrique $g$, il est naturel de consid\'erer la mati\`ere comme un  fluide parfait sans pression. 
La premi\`ere chose \`a faire est de donner une d\'efinition de l'\'energie d'un fluide
parfait sans pression par rapport \`a un observateur. Puisque les particules sont ind\'ependantes deux \`a deux, chaque particule va se comporter comme une seule particule dans le vide, mod\`ele que l'on a d\'ej\`a \'etudi\'e en relativit\'e restreinte.  On va voir qu'en relativit\'e g\'en\'erale, on peut tout calculer en un
point $x \in \M$ et dans l'espace tangent correspondant $T_x \M$. Or
lorsqu'on travaille sur $T_x \M$ muni de la m\'etrique $g_x$, on est
exactement dans le cadre de la relativit\'e restreinte, cadre sur lequel
on va donc s'appuyer pour construire la
th\'eorie. 

On commence par donner la d\'efinition suivante. 

\begin{definition}
Soit $F$ un fluide, $D$ un observateur et $x \in D$. La {\it vitesse du fluide
par rapport \`a $D$} est le vecteur vitesse de la courbe du fluide passant
par $x$ (qui peut \^etre consid\'er\'ee comme un observateur) par rapport
\`a $D$.   
\end{definition} 

\noindent On a dit plus haut que la vitesse d'un observateur par rapport \`a un
autre n'avait pas de sens en relativit\'e g\'en\'erale, parce que pour la d\'efinir, il faudrait d\'ecomposer un vecteur d'un espace tangent $T_y \M$
dans un autre espace tangent $T_x \M$. Lorsque les deux observateurs sont
au m\^eme point, il n'y a plus ce probl\`eme et on peut proc\'eder comme en
relativit\'e restreinte. Donc pour pr\'eciser la d\'efinition ci-dessus,
notons $i_D$ le vecteur unitaire tangent \`a $D$ {\bf en x} et orient\'e
positivement. Notons $\overrightarrow{u}$ le champ de vecteurs unitaire associ\'e \`a
$F$. On \'ecrit de mani\`ere unique
$$\overrightarrow{u}(x) = \overrightarrow{k} + a i_D$$
o\`u $\overrightarrow{k} \in [i_D]^\perp$ (l'orthogonalit\'e \'etant bien \'evidemment relative \`a $g$) et $a \in \M$. Comme en relativit\'e
restreinte, on d\'efinit 

$$\overrightarrow{v}_{F/D} := \frac{\overrightarrow{k}}{a}.$$ 
\smallskip

\noindent 
On consid\`ere un fluide parfait sans pression $F$ de densit\'e de masse $\rho$ dans un domaine $\Om \subset \M$, $x \in \Om$ et un observateur au point $x$ dirig\'e par $i_D$ unitaire et orient\'e positivement. On a vu que la densit\'e d'\'energie de $F$ vue par $D$ en $x$ \'etait 
$$e_D(x) = \frac{\rho(x)}{1-v^2}$$ 
ou encore d'apr\`es \eref{gii}, 
$$e_D(x) = \rho(x) \left(g_x(\overrightarrow{u}(x),i_D)\right)^2.$$
Il faut noter que cette expression n'aurait pas de sens en relativit\'e g\'en\'erale en un autre point que $x$.
Cela conduit \`a d\'efinir la forme bilin\'eaire sym\'etrique 
\[ \tau_x \left| \begin{array}{ccc}
T_x \M \times T_x \M & \to& \mR \\
(\overrightarrow{v},\overrightarrow{w}) & \mapsto &  \rho(x) g_x(\overrightarrow{u}(x),\overrightarrow{v})g_x(\overrightarrow{u}(x),\overrightarrow{w}).
\end{array} \right. \]
En consid\'erant $\tau_x$ pour tout $x$, on obtient ainsi un tenseur deux fois covariant qui v\'erifie que pour tout observateur $D$ au point $x$ dirig\'e par $i_D$, unitaire et orient\'e positivement, on a 
$$\tau(i_D,i_D) = e_D(x)$$ 
o\`u $e_D(x)$ est d\'efini comme ci-dessus. 
D'autre part, on pourrait faire pour l'impulsion la m\^eme construction que pour la densit\'e d'\'energie en relativit\'e restreinte et avec la m\^eme d\'emarche obtenir en relativit\'e g\'en\'erale un vecteur au point $x$ que l'on notera 
$\overrightarrow{P}_D(x)$ et qui est l'analogue de $e_D(x)$ construit ci-dessus. Alors, on peut v\'erifier que si $\overrightarrow{w}$ est unitaire (i.e. $g(\overrightarrow{w},\overrightarrow{w})=1$) et $g$-orthogonal \`a $i_D$,  
$$\tau_x(\overrightarrow{w},\overrightarrow{w}) = g(\overrightarrow{P}_D(x),\overrightarrow{w}).$$ 
Cela justifie de donner la d\'efinition suivante  
\begin{definition}
Le tenseur $\tau$ est appel\'e {\it tenseur d'\'energie-impulsion} associ\'e au fluide $F$.
\end{definition}
   
\noindent En fait, on confondra souvent $\tau$ avec le tenseur deux fois contravariant qui lui est associ\'e. On le notera toujours $\tau$. 
Il est clair que 
\begin{eqnarray} \label{def_tei}
\tau = \rho \overrightarrow{u} \otimes \overrightarrow{u},
\end{eqnarray} 
$\overrightarrow{u}$ \'etant le champ de vecteurs unitaires associ\'e \`a $F$. 
On montre maintenant que 

\begin{prop} \label{tei_geod}
Le champ de vecteurs de composantes $\nabla_k \tau^{ kl}$ est le champ de vecteurs nul si et seulement si les courbes du fluides sont des g\'eod\'esiques de $(\M,g)$.
\end{prop}

\begin{proof} 
Soit $x \in \Om$. Prenons une base $(e_1,\cdots,e_4)$ de $T_x \M$ tel que $\nabla e_j(x) = 0$. On a alors 
\begin{eqnarray*}
\begin{aligned} 
\nabla_k\tau^{kl} e_l & = \nabla_k (\rho \overrightarrow{u}^k \overrightarrow{u}^l) e_l \\
& = \nabla_k(\rho \overrightarrow{u}^k) \overrightarrow{u}^l e_l  + \rho \overrightarrow{u}^k \nabla_k \overrightarrow{u}^l e_l  \\
& = \div(\rho \overrightarrow{u}) \overrightarrow{u} + \rho D_{\overrightarrow{u}} \overrightarrow{u}   
\end{aligned}
\end{eqnarray*}
o\`u $D$ est la connexion de Levi-Civita associ\'ee \`a $g$.  
Compte-tenu de \eref{div=0}, le champ de vecteurs $\nabla_k\tau^{kl} e_l $ est nul si et seulement si $D_{\overrightarrow{u}} \overrightarrow{u} =0$, c'est-\`a-dire si et seulement si les courbes du fluide sont des g\'eod\'esiques de $(\M,g)$. 
\end{proof}

\subsection{L'\'equation d'Einstein} \label{eq_ein}
Nous sommes maintenant en mesure de d\'efinir le deuxi\`eme axiome de la
relativit\'e g\'en\'erale. Pour cela, la strat\'egie est d'essayer de
copier les mod\`eles de la relativit\'e restreinte et de la m\'ecanique
classique pour retrouver les lois de Newton pour des  vitesses faibles par
rapport \`a celle de la lumi\`ere.  Le probl\`eme est que dans ces
mod\`eles, les observateurs galil\'eens jouent un r\^ole fondamental (par
exemple pour d\'efinir l'\'energie et l'impulsion) et qu'en relativit\'e
g\'en\'erale, on n'a aucun observateur privil\'egi\'e. L'id\'ee est alors
de supposer dans un premier temps qu'on a de tels observateurs ("presque"
galil\'eens), de trouver une bonne formulation du deuxi\`eme axiome
dans ce cadre et de voir que ce que l'on a trouv\'e est en fait
intrins\`eque et ne d\'epend pas de ces observateurs particuliers.

\begin{definition}
Un {\it domaine statique} est un ouvert connexe $\Om$ de $\M$ tel que 
$(\Om,g)$ est isom\'etrique \`a $(I \times \om,\tilde{g})$ o\`u $I$ est un
intervalle ouvert, $\om$ est un ouvert connexe de $\mR^3$ et o\`u dans la
``carte canonique'' (donn\'ee par l'isom\'etrie de $\Om$ dans $I \times
\om$), la matrice de $g$ au point $(t,x^1,x^2,x^3) \in I \times \om$ est
donn\'ee par 
\[ \left( \begin{array}{cccc}
-f^2(x^1,x^2,x^3) & 0 &0 &0 \\
0 &\bar{g}_{11} (x^1,x^2,x^3) & \bar{g}_{12} (x^1,x^2,x^3) &\bar{g}_{13} (x^1,x^2,x^3) \\
0 &\bar{g}_{21} (x^1,x^2,x^3)  &  \bar{g}_{22} (x^1,x^2,x^3)
&\bar{g}_{23} (x^1,x^2,x^3) \\
0 &\bar{g}_{31} (x^1,x^2,x^3)  &  \bar{g}_{32} (x^1,x^2,x^3)
&\bar{g}_{33} (x^1,x^2,x^3) 
\end{array} \right) \]
o\`u $f,\bar{g}_{i,j}: \om \to \mR$ ($i \in \{1,2,3\}$) sont des fonctions suffisamment
r\'eguli\`eres pour que la suite ait un sens. Autrement dit, on a 
$$\tilde{g}(t,x^1,x^2,x^3)= -f^2(x^1,x^2,x^3) + \sum_{k,l=1}^3
  \bar{g}_{kl}(x^1,x^2,x^3) dx^k\otimes dx^l $$ pour tout
$(t,x^1,x^2,x^3) \in I \times \om$,
\end{definition}

\begin{remark}
Le mot {\it statique} provient du fait que les composantes de la m\'etrique
ne d\'ependent pas du temps. 
\end{remark}

\noindent Il n'y a aucune raison qu'il existe dans $(\M,g)$ un domaine
statique mais Hawking a montr\'e que l'existence d'un tel domaine statique
est \'equivalente \`a l'existence d'un {\it champ de Killing $K$}
(c'est-\`a-dire un champ de vecteurs engendrant un flot isom\'etrique) et
d'une hypersurface $H$ de genre espace orthogonale en tout point au champ
de vecteurs $K$. Le flot  de $K$ donne donc un groupe  \`a un param\`etre
d'isom\'etries dont les orbites (ce sont les courbes
int\'egrales de $K$) sont $g$-orthogonales \`a $H$.  \\

\noindent Autrement dit, consid\'erons une courbe int\'egrale de $K$. Cette
courbe d\'efinit un observateur $D$. L'espace vu par $D$ est \`a tout
instant isom\'etrique \`a $(H,\bar{g})$ o\`u $\bar{g}$ est la restriction de
$g$ \`a $H$. Ces observateurs observant toujours le m\^eme espace sont les
analogues des observateurs galil\'eens en m\'ecanique classique ou en
relativit\'e restreinte. \\

\noindent Nous travaillons donc dans $(I \times \om, \tilde{g})$. La
mati\`ere est mod\'elis\'ee par un fluide parfait sans pression $F$. Rappelons
que le premier axiome dit que les courbes de $F$ param\'etr\'ees par le
temps sont des g\'eod\'esiques de $(\M,g)$. L'id\'ee fondamentale pour
trouver le deuxi\`eme axiome est que le fluide lui-m\^eme (y compris sa
densit\'e de masse) est li\'e \`a la m\'etrique $g$. C'est ce lien que l'on
cherche \`a d\'eterminer de mani\`ere \`a retrouver (en approximation) les
lois de Newton. Remarquons que les domaines statiques permettent de se ramener \`a la relativit\'e restreinte ou \`a la m\'ecanique de la fa\c{c}on suivante : prenons le cas 
o\`u 
 $f= 1$ et o\`u
$\bar{g}$ est la m\'etrique euclidienne. Alors, on a deux mani\`eres de voir les choses. Soit on travaille dans $(I \times \om, \tilde{g})$. Dans ce cas, on retrouve l'espace-temps de la
relativit\'e restreinte. Sinon, on peut travailler dans $I \times w$ muni
de la forme  $T$ associ\'ee \`a la forme bilin\'eaire $-f^2 dt^2$  dont le noyau (tangent \`a $w$) est muni du produit scalaire $\bar{g}$. On retrouve le mod\`ele de la m\'ecanique classique. {\bf Pour trouver le deuxi\`eme axiome, c'est avec cette vision des choses que l'on va travailler.}  \\

\noindent Pour cela, consid\'erons une courbe $\C$ du fluide param\'etr\'ee par $c$,
param\'etrisation normale positive. On \'ecrit $c$ dans la carte
$(t,x^1,x^2,x^3,x^4)$: 

$$c(s) = (c_0(s),c_1(s),c_2(s),c_3(s))$$ 
et on note $\overrightarrow{k}(s)= (c_1'(s),c_2'(s),c_3'(s))$ et $\overrightarrow{a}(s)=(c_1''(s), c_2''(s),c_3''(s))$
\noindent \'Ecrire $c$ dans cette carte revient \`a regarder $\C$ \`a
travers 
les yeux d'un observateur galil\'een $D$, c'est-\`a-dire une courbe
int\'egrale de $K$. Le vecteur $\overrightarrow{k}$ repr\'esente la vitesse de $\C$
par rapport \`a $D$ et le vecteur $\overrightarrow{a}$ est le vecteur acc\'el\'eration
de $\C$ vue par $D$.  En fait, pour travailler vraiment avec le mod\`ele de
la m\'ecanique 
classique, il faudrait que $c$ soit une param\'etrisation normale pour $T= -\frac{\partial}{\partial t}$ et donc normaliser $c$ pour avoir $c_0' = 1$. 
Cependant, comme $g_{c(s)} (c'(s),c'(s))= -1$, on a 
$$ - f^2 c_0'(s) + \| \overrightarrow{k}(s) \|_{\bar{g}}= -1$$ 
c'est-\`a-dire 
\begin{eqnarray} \label{x0'}
(c_0')^2= \frac{1+\| \overrightarrow{k} \|_{\bar{g}}}{f^2}.  
\end{eqnarray} 
Maintenant, on doit supposer qu'on est proche du mod\`ele de la m\'ecanique classique donc $f$ est proche de $1$. D'autre part, les vitesses sont suppos\'ees petites par rapport \`a la vitesse de la lumi\`ere, c'est-\`a-dire qu'on suppose 

\begin{eqnarray} \label{vo1}
 \| \overrightarrow{k} \|_{\bar{g}} = o(1)
\end{eqnarray} 
ce qui fait que 
\begin{eqnarray} \label{xo0}
c_0' =o(1)
\end{eqnarray} 
(autrement dit, $c$ est "presque" une param\'etrisation normale pour la forme $-dt^2$) et $\overrightarrow{a}$ peut bien \^etre consid\'er\'e comme le vecteur acc\'el\'eration de la courbe $\C$. 
M\^eme si cette approche n'a rien de rigoureux, elle est physiquement coh\'erente. Elle va nous permettre de trouver par l'intuition un bon deuxi\`eme axiome dont la validit\'e sera v\'erifi\'ee par les observations physiques. \\

\noindent La premi\`ere chose \`a faire est de se d\'ebrouiller pour retrouver la relation \eref{Relation2} en approximation. Pour cela, nous devons exprimer $\overrightarrow{a}$ en fonction de donn\'ees g\'eom\'etriques. C'est l'objet du r\'esultat suivant : 
\noindent   
 \begin{prop} \label{geod} Les composantes du vecteurs $\overrightarrow{a}$
   v\'erifient pour tout $k \in \{1,2,3\}$
$$c_k ''(s) = - \frac{1}{f} ( 1+ v^2) \nabla^k_{\bar{g}} f  - \sum_{i,j=1}^3
\Gamma_{ij}^k c_i'
c_j' $$
o\`u l'on a pos\'e $v = \| \overrightarrow{k} \|_{\bar{g}}$ et o\`u  $
\nabla^k_{\bar{g}} f = g^{kj}\partial_j f$. 
\end{prop}
 
\begin{proof}
D'apr\`es l'axiome $1$, $\C$ est une g\'eod\'esique de
$(\M,g)$. Autrement dit, 
$D_{c'(t)}c'(t) =0$ ($D$ est la d\'eriv\'ee covariante associ\'ee \`a la
connexion de Levi-Civita de $g$).  
Donc 
\begin{eqnarray} \label{dctct}
\begin{aligned} 
0 = D_{c'(t)}(c'(t)) &  
= D_{c'(t)}\left(c_0' \frac{\partial}{\partial t}\right) +\sum_{i,j=1}^3  D_{c'(t)}
\left(c_i' \frac{\partial}{\partial x^i}\right)  \\
& = c''(t) + c_0' D_{c'(t)} \frac{\partial}{\partial t} + \sum_{i,j=1}^3
c_i'  D_{c'(t)}
\frac{\partial}{\partial x^i}.
\end{aligned}
\end{eqnarray} 

\noindent \'Ecrivons maintenant que 
$$D_{c'(t)} \frac{\partial}{\partial t}= c_0' D_{\frac{\partial}{\partial
    t}}  \frac{\partial}{\partial t} +  \sum_{i=1}^3 c_i'  D_{\frac{\partial}{\partial
    x^i}}   \frac{\partial}{\partial t}.$$
Comme d'apr\`es \eref{christoffel}, pour $i,k \geq 1$, $\Gamma_{i0}^k = 0$,
on obtient 
\begin{eqnarray} \label{ddt}
D_{c'(t)} \frac{\partial}{\partial t} = c_0' \Gamma_{00}^{k} \frac{\partial}{\partial
    x^k}. 
\end{eqnarray} 
De m\^eme, on calcule que 
\begin{eqnarray} \label{ddk}
D_{c'(t)} \frac{\partial}{\partial x^i } = c_0' \Gamma_{0i}^0 +  \sum_{j=1}^3
c_j' \Gamma_{ij}^k \frac{\partial}{\partial x^k }.
\end{eqnarray}

\noindent En rempla\c{c}ant dans \eref{dctct} et en regardant composante par
composante, on trouve que pour tout $k \in \{1,2,3\}$ 
$$0 = c_k'' + \Gamma_{00}^k (x_0')^2 + \sum_{i,j=1}^3   \Gamma_{ij}^k x_i'
x_j'.$$
En utilisant \eref{christoffel} (voir plus bas) et \eref{x0'}, la preuve de la
proposition est compl\`ete. 
\end{proof}

\noindent Revenons maintenant \`a ce qui nous int\'eresse: retrouver la relation \eref{Relation2}. Pour \^etre "approximativement" en m\'ecanique classique, on doit supposer que $f$ vaut "presque" $1$ et $\bar{g}$ est presque la m\'etrique euclidienne.   On \'ecrit donc $f = 1 + h$ et on suppose que $h=o(1)$. On \'ecrit aussi $\bar{g}_{ij}=\delta_{ij} +o(1)$. 
Compte-tenu de \eref{christoffel} et de \eref{vo1}, on obtient que 
\begin{eqnarray} \label{veca}
\overrightarrow{a}= - \overrightarrow{\nabla h} (1 + o(1)).
\end{eqnarray}
Pour retrouver la relation \eref{Relation2}, on voudrait que cette fonction $h$ soit le potentiel newtonnien $f$ qui appara\^{\i}t dans la relation \eref{Relation2}, c'est-\`a-dire, en vertu de la relation \eref{Relation1} que 

\begin{eqnarray} \label{cequonveut}
\Delta_{\bar g} h \equiv 4 \pi \rho.  
\end{eqnarray} 
Or on calcule que 

\begin{prop} \label{ricci}
Dans la carte $(t,x^1,x^2,x^3)$, la courbure de Ricci de $g$ v\'erifie 
$$Ric \left(\frac{\partial}{\partial t}, \frac{\partial}{\partial t}\right) = -f
\Delta_{\bar{g}} f.$$
\end{prop}

\begin{proof}
On commence par calculer les symboles de Christoffel de la connexion de
Levi-Civita associ\'ee \`a $g$ dans cette
carte. Nous notons ``$0$'' la coordonn\'ee associ\'ee \`a $t$ et utilisons
les conventions d'Einstein. 
On rappelle que par d\'efinition, on a pour tous $i,j,k \in \{0,1,2,3\}$ 
$$\Gamma_{ij}^k= \frac{1}{2} g^{kl}\left(\partial_i g_{lj} +\partial_j g_{li}-
\partial_l g_{ij}\right)$$
et que $\Gamma_{ij}^k= \Gamma_{ji}^k$. 
Donc si $i,k \in \{1,\cdots,3\}$,
\begin{eqnarray} \label{christoffel} 
\Gamma_{00}^0 = \Gamma_{i0}^k = 0 \; ; \; \Gamma_{i0}^0= \frac{1}{f}
\partial_i f \, \hbox{ et }  \, \Gamma_{00}^k = f g^{kl} \partial_l f.
\end{eqnarray}
Maintenant, on sait que la courbure de Ricci s'exprime dans une carte en
fonction 
des symboles de Christoffel (voir l'appendice)
gr\^ace  \`a la formule suivante:
$$R_{\al \beta} = \partial_i\Gamma_{\al \beta}^i- \partial_\beta
\Gamma_{\alpha i }^i+ \Gamma_{im}^i \Gamma_{\alpha \beta}^m - \Gamma_{\beta
  m}^i \Gamma_{i \alpha}^m$$
o\`u l'on a not\'e $R_{\al\be}$ les composantes du tenseur de Ricci dans
la carte consid\'er\'ee. 
On a donc: 
$$R_{00} = \partial_i \Gamma_{00}^i - \partial_0 \Gamma_{i0}^i
 + \Gamma_{00}^m \Gamma_{mi}^i -\Gamma_{i0}^m\Gamma_{m0}^i$$
c'est-\`a-dire en rempla\c{c}ant les symboles de Christoffel par leur valeur:
\begin{eqnarray} \label{R00} 
R_{00}= \partial_i (f g^{ij} \partial_j f) + f \Gamma_{mi}^i g^{mj}
\partial_j f - 2 |df|_{\bar{g}}^2.
\end{eqnarray}  
On a utilis\'e le fait que puisque 
$f$ ne d\'epend pas de $t$, $|df|_{\bar{g}} = |df|_{ g}$.  
Maintenant, on calcule 

$$
\partial_i (f g^{ij} \partial_j f)=   |df|^2_{\bar{g}}  
+ f \partial_i( g^{ij}) \partial_j f + f g^{ij} \partial_{ij} f.$$
En \'ecrivant que $\nabla_i g^{-1} = 0$, on obtient 
 
$$\partial_i g^{ij} = -g^{mj} \Gamma^i_{im} - g^{im} \Gamma^j_{im}$$
ce qui donne 
$$\partial_i (f g^{ij} \partial_j f)=   |df|^2_{\bar{g}} - f( \partial_j f)
g^{mj} \Gamma^i_{im} - f (\partial_j f) g^{im} \Gamma^j_{im} + f g^{ij} \partial_{ij} f.$$
Rappelons que l'{\it op\'erateur de D'alembert} ou {\it d'alembertien} est
l'analogue riemannien du laplacien. Il est d\'efini en coordonn\'ees par 
$$\square_g f = - g^{ij} \left(\partial_{ij} f -\Gamma_{ij}^k \partial_k f\right).$$
Il est fr\'equent de conserver les notations riemanniennes en g\'eom\'etrie
lorentzienne (par exemple en ce concerne les courbures) mais pour le
d'alembertien, une notation diff\'erente indique la diff\'erence de nature
entre ces deux op\'erateurs : le laplacien est elliptique alors que le
d'alembertien est hyperbolique. \\

\noindent Reprenons notre calcul. Nous obtenons ainsi  
$$\partial_i (f g^{ij} \partial_j f)=   |df|^2_{\bar{g}} - f (\partial_j f)
g^{mj} \Gamma^i_{im} - f \square_g f.$$
En rempla\c{c}ant dans \eref{R00}, on a 
\begin{eqnarray} \label{R002} 
R_{00} = - f \square_g f - |df|^2_{\bar{g}}.
\end{eqnarray} 

\noindent On calcule maintenant  en utilisant que $g^{0i} = -f^2
\delta^{0i}$ et que $\Gamma_{00}^m = \frac{1}{f} \partial_m f$
\begin{eqnarray*} 
\begin{aligned} 
f \square_g f & = - g^{ij} \left(f \partial_{ij} f - f \Gamma_{ij}^k \partial_k f\right)\\
& = f \Delta_{\bar{g}}f + g^{00} \Gamma_{00}^m \partial_m f \\
& =  f
\Delta_{\bar{g}}f - |df|^2_{\bar{g}}.
\end{aligned}
\end{eqnarray*} 
En revenant \`a \eref{R002}, on obtient 
$$R_{00} = - f \Delta_{\bar{g}} f$$
ce qui termine la d\'emonstra--tion de la proposition \ref{ricci}. 
\end{proof}

\noindent De cette proposition, on d\'eduit puisque $f= 1+h$ avec $h=o(1)$,
que $\Delta_{\bar{g}} h = - Ric \left(\frac{\partial}{\partial t} , \frac{\partial}{\partial t}\right)(1+o(1))$. En revenant aux notations tensorielles, pour avoir la relation \eref{cequonveut}, il faut donc imposer 
$$R_{00} = 4 \pi \rho.$$
Rappelons que le tenseur d'\'energie-impulsion est d\'efini par 
$$\tau = \rho \overrightarrow{u} \otimes \overrightarrow{u}$$
et comme $c'=\overrightarrow{u}$, 
$$\tau_{00} = \rho (c_0')^2 = \rho (1+o(1))$$
en utilisant \eref{xo0}.
La relation \eref{cequonveut} est donc satisfaite si on suppose que 
$$Ric = 4 \pi \tau.$$

\begin{remark}
Il suffit que cette relation soit satisfaite sur la composante $00$ pour avoir \eref{cequonveut} mais pour avoir une \'equation intrins\`eque, on l'impose comme \'etant une \'egalit\'e tensorielle. 
\end{remark}

\noindent Malheureusement, la proposition \ref{tei_geod} implique que l'on doit avoir 
$$\nabla^i R_{ij} =0$$
ce qui n'a aucune raison d'\^etre vrai en g\'en\'eral. Par contre, on remarque que le tenseur dit {\it tenseur d'Einstein}
\begin{eqnarray} \label{eij}
E_{ij} := R_{ij} - \frac{1}{2} R g_{ij}
\end{eqnarray}
(o\`u $R:= g^{kl} R_{kl}$ est la courbure scalaire) v\'erifie cette condition (i.e. $\nabla^i E_{ij} =0$). Cette relation se d\'eduit de l'identit\'e de Bianchi sur la courbure de Riemann. D'o\`u l'id\'ee de poser 
\begin{eqnarray} \label{einstein1}
Ric -\frac{1}{2} R g = 8 \pi \tau.
\end{eqnarray}
Cette relation conserve la relation \eref{cequonveut}. Pour le voir, contractons chaque c\^ot\'e de l'\'egalit\'e par $g^{ij}$. Comme $g^{ij}g_{ij} =4$ et comme $\tau_{ij} = \rho c_i' c_j'$ (car $\overrightarrow{u} = c'$), on obtient $R-2R= 8 \pi \rho c_i' c_j' g^{ij}$. Mais d'apr\`es \eref{vo1}, $c_i' = o(1)$ si $i \not= 0$ et on a d\'ej\`a vu que $c_0'=1+o(1)$ (d'apr\`es \eref{xo0}). On obtient donc 
$-R = - 8\pi \rho(1+o(1))$ puisque $g^{00} = -1 + o(1)$. On en d\'eduit en utilisant l'\'equation \eref{einstein1} que

\begin{eqnarray*} 
\begin{aligned}
\Delta_{\bar{g}} h&  = - (1+o(1)) R_{00} = - (1+o(1))( \frac{1}{2} R g_{00} + 8 \pi \tau_{00}) \\
& = -(1+o(1)) (- 4 \pi \rho +8\pi \rho) \\
& = - (1+o(1))  4 \pi \rho 
\end{aligned}
\end{eqnarray*}
et on retrouve \eref{cequonveut}. \\

\noindent 
Si on ajoute \`a $E_{ij}$ un terme de la forme $\Lambda g$ (o\`u $\Lambda$
est un r\'eel), on garde la relation \eref{eij}. Par contre, on perd la
relation \eref{cequonveut}.  Prendre $\Lambda =0$ conduit \`a consid\'erer
que l'univers est en expansion (voir le chapitre suivant)  
ce que refusait compl\`etement Einstein. 
C'est pourquoi il a ajout\'e ce terme. L'\'equation garde sa coh\'erence (i.e. la relation \eref{eij}). On peut m\^eme montrer que c'est le seul terme que l'on peut ajouter pour garder la coh\'erence de l'\'equation. Si l'on suppose que $\Lambda$ est petit, on est proche des lois de Newton.  
C'est cette \'equation intrins\`eque (on n'a pas besoin d'avoir de domaine statique pour la consid\'erer) que l'on gardera. \\ 
 
\noindent Autrement dit, on est maintenant en mesure d'\'enoncer les deux axiomes qui r\'egissent la mati\`ere lorsqu'on consid\`ere qu'il n'y a qu'un fluide parfait pression dans l'espace-temps : \\

\noindent \underline{Axiome 1 :} Les courbes du fluide param\'etr\'ees par leur temps propre sont des g\'eod\'esiques de $(\M,g)$.\\ 

\noindent \underline{Axiome 2 :} La mati\`ere et la m\'etrique sont li\'ees par l'\'equation d'Einstein 
\begin{eqnarray} \label{einstein}
Ric -\frac{1}{2} R g = 8 \pi \tau - \Lambda g
\end{eqnarray} 
o\`u $\Lambda$ est une constante, appel\'ee {\em constante cosmologique}, que l'on peut choisir.\\

\begin{remark}
Si la mati\`ere n'est pas mod\'elis\'ee par un fluide parfait, on conserve tout de m\^eme ces deux axiomes sous cette forme mais c'est le tenseur d'\'energie-impulsion qui prendra une autre forme. 
\end{remark}
\noindent Consid\'erons seulement l'axiome $2$. La relation \eref{eij} implique que $\nabla^i \tau_{ij} = 0$  et la proposition \ref{tei_geod} implique alors l'axiome $1$. {\bf On consid\`ere donc que le comportement de la mati\`ere est r\'egi par l'axiome $2$ seulement}. L'axiome $1$ est alors automatiquement vrai. \\

\noindent {\bf \`A partir de maintenant et dans tous les chapitres qui suivent, nous nous pla\c{c}ons toujours
  dans l'espace-temps de la relativit\'e g\'en\'erale et nous supposerons
  que le comportement de la mati\`ere est r\'egi par l'axiome 2 seulement. } \\

\begin{remark}
Vers la fin de sa vie, Einstein a admis qu'il avait fait une erreur en
refusant d'admettre que  $\Lambda=0$. En tout cas, si la constante
$\Lambda$ n'est pas nulle, elle doit \^etre tr\`es petite. En effet, ce
terme additionnel d\'etruit  la
relation \eref{cequonveut}. Par cons\'equent, on ne retrouve plus les lois
de la m\'ecanique classique \`a petite \'echelle.  Les 
mesures physiques r\'ecentes
tendent \`a montrer que $\Lambda>0$ est non nulle, petite, mais pas aussi
petite que ce que l'on
pensait. \\
\end{remark}

\begin{remark}
La constante cosmologique fournit \'egalement une explication possible \`a l'{\'energie noire}. 
\end{remark} 

\begin{remark}
Pour \'etablir l'\'equation d'Einstein, on a travaill\'e avec la courbure de Ricci mais on aurait pu penser \`a poser $R= - 4 \pi \rho$ ce qui est a priori suffisant pour avoir la relation \eref{cequonveut}. Physiquement, cela n'aurait pas pu mod\'eliser correctement la r\'ealit\'e car cette relation faisait intervenir seulement la densit\'e de masse et pas les courbes du fluide. On aurait aussi pu penser \`a utiliser le tenseur de Riemann mais les \'equations auraient \'et\'e beaucoup plus compliqu\'ees et les observations physiques montrent que le choix de l'axiome $2$ ci-dessus est un mod\`ele tr\`es proche de la r\'ealit\'e. 
\end{remark}

\providecommand{\bysame}{\leavevmode\hbox to3em{\hrulefill}\thinspace}

\end{document}